\documentclass[a4paper, 10pt]{article}

\usepackage{packageList}
\usepackage{macros}

\pgfplotsset{compat=1.11}

\title{A novel class of stabilized greedy kernel approximation algorithms: Convergence, stability \& uniform point distribution}
\author[1]{Tizian Wenzel \thanks{tizian.wenzel@mathematik.uni-stuttgart.de}}
\author[1,2]{Gabriele Santin \thanks{gsantin@fbk.eu, \href{http://orcid.org/0000-0001-6959-1070}{orcid.org/0000-0001-6959-1070}}}
\author[1]{Bernard Haasdonk \thanks{haasdonk@mathematik.uni-stuttgart.de}}
\affil[1]{Institute for Applied Analysis and Numerical Simulation, University of Stuttgart, Germany}
\affil[2]{Center for Information and Communication Technology, Fondazione Bruno Kessler, Italy}

\begin{document}

\maketitle 
\begin{abstract}
Kernel based methods provide a way to reconstruct potentially high-dimensional functions from meshfree samples, i.e., sampling points and corresponding target 
values. 
A crucial ingredient for this to be successful is the distribution of the sampling points. 
Since the computation of an optimal selection of sampling points may be an infeasible task, one promising option is to use greedy methods. 

Although these methods may be very effective, depending on the specific greedy criterion the chosen points might quickly lead to instabilities in the 
computation. 
To circumvent this problem, we introduce and investigate a new class of \textit{stabilized} greedy kernel algorithms, which can be used to 
create a scale of new selection strategies. 

We analyze these algorithms, and in particular we prove convergence results and quantify in a precise way the distribution of the selected points. These 
results allow to 
prove, in the case of certain Sobolev kernels, that the algorithms have optimal stability and optimal convergence rates, including for functions outside the 
native space 
of the kernel. 
The results also apply to the case of the usual $P$-greedy algorithm, significantly improving state-of-the-art results available in the literature. 
Illustrative experiments are presented that support the theoretical findings. %
\end{abstract}

\section{Introduction}\label{sec:introduction}

We start with introducing our central terminology and notation, which is well-known in kernel approximation literature, cf.\ \cite{Fasshauer2007, Fasshauer2015, Wendland2005}. For a nonempty set $\Omega \subset \mathbb{R}^d$ a real valued kernel is defined as a symmetric function $k: \Omega \times \Omega \rightarrow \mathbb{R}$. 
The \textit{kernel matrix} $A \in \mathbb{R}^{N \times N}$ associated to this kernel and to an arbitrary set of points $X_N = \{x_1, .., x_N\} \subset \Omega$ is given by entries
$A_{ij} = k(x_i, x_j)$. In the following the sets $X_N$ will always be assumed to consist of pairwise distict points. A kernel $k$ is called \textit{strictly positive definite} if the kernel matrix is positive definite for any such $X_N 
\subset \Omega$ 
and any $N \in \mathbb{N}$.

For every strictly positive definite kernel on $\Omega$ there is a unique \textit{native space} $\ns$ with inner product $(\cdot, 
\cdot)_{\ns}$, i.e., a Hilbert space of functions where the kernel $k$ acts 
as a reproducing kernel, that is
\begin{enumerate}
\item $k(\cdot, x) \in \ns ~ \forall x \in \Omega$,
\item $(f, k(\cdot, x))_{\ns} = f(x) ~ \forall f \in \ns, \forall x \in \Omega$.
\end{enumerate}

Strictly positive definite kernels can be used to interpolate functions. For this, consider a function $f \in \ns$. Then there is a unique interpolant $s_N$ of $f$ on $X_N$, which is given by the orthogonal projection $\Pi_{V(X_N)}(f)$ of $f$ to the subspace $V(X_N) := \text{span}\{ k(\cdot, x_i), x_i \in X_N \}$. Additionally this interpolant is the minimum norm interpolant among all interpolants from 
$\ns$. It can be expressed as
\begin{align*}
s_N(\cdot)=\Pi_{V(X_N)}(f) = \sum_{i=1}^N \alpha_i k(\cdot, x_i),
\end{align*}
where the coefficients $\alpha = ( \alpha_i )_{i=1}^N \in \R^N$ can be computed by solving the linear equation system
\begin{align} \label{eq:linear_eq_system}
A \alpha = b,
\end{align}
with the kernel matrix $A$ of $k$ on $X_N$, and the vector $b = (f(x_i))_{i=1}^N$ of function values. Since the kernel is assumed to be strictly positive 
definite, 
the kernel matrix is positive definite and thus Equation \eqref{eq:linear_eq_system} possesses a unique solution.

A priori it is usually unclear how to choose an appropriate set of sampling points, although it is evident that this choice strongly influences the results of the interpolation process. To overcome this problem, often greedy methods are applied \cite{DeMarchi2005, Mueller2009, SchWen2000, Wirtz2013, SH16b}. 

Greedy algorithms are largely studied and applied in the fields of approximation and computational mathematics. We refer for example to  
\cite{Temlyakov2008} for a general treatment of greedy methods. Convergence 
analysis for several greedy algorithms, also in reduced basis methods, can be found in \cite{Binev2011, BM12, 
H13, Mula2016}.

In the present setting of kernel interpolation, the greedy algorithms start with an empty set $X_0 
:= \emptyset$ and then they iteratively add another interpolation point as $X_{N+1} := X_{N} \cup \{ x_{N+1} \}$ at each step which is optimal with regard to 
some selection criterion $\eta^{(N)}$:
\begin{align}\label{eq:greedy_selection}
x_{N+1} = \argmax_{x \in \Omega} \eta^{(N)}(x).
\end{align}
These greedy algorithms stop if the interpolant is exact or if some predefined stopping criterion is met, and they are practicable and fast.

In the greedy kernel literature there are three main selection criteria, namely $f$-greedy, $P$-greedy and $f/P$-greedy \cite{SchWen2000, DeMarchi2005, 
Mueller2009}, 
that we will recall in the following. Depending on the chosen selection criterion different interpolation point distributions may result, and they lead to 
different 
effects concerning convergence speed or stability of the numerical scheme. 

The convergence is closely linked to the \textit{fill distance} $h_{X_N, \Omega}$ which describes the largest ball which can be centered 
in $\Omega$ without intersecting any interpolation point $x_i \in X_N$. The \textit{separation distance} $q_{X_N}$ measures instead  the distance of the two 
closest 
interpolation 
points and it can be related to the stability of the interpolation procedure via the condition number of the kernel matrix. Thus we define
\begin{equation}\label{eq:fill_and_sep_dist}
\begin{aligned}
h_N &:=& h_{X_N, \Omega} &:= \sup_{x \in \Omega} \min_{x_i \in X_N} \Vert x-x_i \Vert_2, \\
q_N &:=& q_{X_N} &:= \min_{x_i \neq x_j \in X_N} \Vert x_i - x_j \Vert_2.
\end{aligned}
\end{equation}
Roughly speaking, a small fill distance implies a small error, while a large separation distance guarantees stability. If a sequence of sets 
of points satisfies $h_N \leq C \cdot q_N$ for some $C > 0$ and all $N \in \mathbb{N}$, then the \textit{uniformity constant} $\rho_N:=h_N/q_N \leq C$ is 
bounded and  
the sequence is called \textit{asymptotically uniformly distributed} and, intuitively, it 
possesses both the desirable properties and it is thus particularly suitable for interpolation.

In this paper we introduce a new class of greedy kernel approximation algorithms, the so called \textit{stabilized} or \textit{restricted greedy algorithms}. 
They are 
obtained as a modification of existing selection strategies by introducing a restriction on the set of admissible points, which is steered by a parameter 
$\gamma\in(0,1]$. This modification creates, for each selection rule, a scale of methods depending on $\gamma$. 

We prove convergence results for the new scale of algorithms by controlling the decay of the Power function associated with the selected points. This result is 
an extension of a result in \cite{SH16b}, and proves that the Power function of the new algorithms decays algebraically or exponentially, depending on known 
convergence rates for interpolation by uniform points.

In the case of certain kernels associated to Sobolev spaces, we also prove that the known decay rates are optimal up to constants, i.e., no better rate of 
decay 
of the Power function can be obtained, even for points that are globally optimized instead than selected iteratively. %

We then use this decay of the Power function to derive three groups of results. First, we obtain a refined version of a bound of \cite{DeMarchi2005} and 
use it in combination with an idea of \cite{SH16b} to obtain a precise decay of the fill distance of the points selected by the algorithms. This result is a 
strict 
improvement of the one of \cite{SH16b}, and in particular it allows to obtain error bounds for the interpolation error w.r.t. a generic $L^p$ norm (and not 
necessarily 
the $L^{\infty}$ norm) and for the approximation of the derivatives of the target function. These general rates of convergence coincide with the worst-case 
optimal ones.

Second, we modify an idea introduced in \cite{DeMarchi2005} to deduce precise lower bounds for the decay of the separation distance. Thanks to this result, we 
can guarantee a control on the condition number of the interpolation matrices and on the Lebesgue constant, and thus on the stability of the algorithms. 

Third, combining the first two results we prove that the selected points are asymptotically uniform. Since these algorithms can be run very efficiently on very 
general 
geometries, they can be used to generate uniform point distributions for other methods such as meshless PDE solvers. Moreover, a result of \cite{Narcowich2006} 
applies 
to the interpolation by these new algorithms, and in particular we prove that they can be used to approximate  with 
optimal rates of convergence functions which are outside the native space of the kernel.

All these results apply to any of the $\gamma$-stabilized greedy algorithms, with constants depending on $\gamma$. In particular, they cover the case of 
$P$-greedy, for which the results of this paper provide new or refined results. The $f$- and $f/P$-greedy selections are also covered. In this case the 
modified algorithms allow to obtain for the first time convergence rates that are at least worst-case optimal.

The paper is organized as follows.
We start by recalling some additional details on kernel theory in Section \ref{sec:kernels}, while in Section \ref{sec:restricted_greedy_alg} the stabilized 
greedy algorithms are introduced. The following Section \ref{sec:conv_rates} gives precise bounds on the decay of the maximal Power function values, and 
these bounds allow to prove in Section \ref{sec:uniformity} that the resulting sampling points are distributed asymptotically uniformly by providing upper 
bounds on the 
fill distance and lower bounds on the separation distance. Using this uniformity, further results are drawn on the stability and convergence of the algorithms 
in Section \ref{sec:extensions}. Section \ref{sec:numerical_experiments} concludes with some numerical experiments which complement the analytic results.

\section{Kernel interpolation and greedy algorithms}\label{sec:kernels}

To begin with we start with some additional background about kernel based approximation. A thorough introduction and more details can be found e.g. in the monographs \cite{Wendland2005, Fasshauer2007, Fasshauer2015} already mentioned in the previous section.

A way to measure the interpolation error $\Vert f - \Pi_{V(X_N)}(f) \Vert_{L^{\infty}}$ is given by the Power function $P_N:=P_{X_N}: \Omega \rightarrow 
\mathbb{R}$, which can be 
defined as the norm of the pointwise error functional, i.e., 
\begin{align}
\label{eq:power_function_via_sup}
P_N(x) =& \sup_{0 \neq f \in \ns} \frac{|f(x) - \Pi_{V(X_N)}(f)(x)|}{\Vert f \Vert_{\ns}},
\end{align}
and it can be proven that it can be expressed in the following way:
\begin{align}\label{eq:power_function_as_norm}
P_N(x) 
=& \Vert k(\cdot, x) - \Pi_{V(X_N)}(k(\cdot, x)) \Vert_{\ns}.
\end{align}
Reshaping Equation \eqref{eq:power_function_via_sup} defining $P_N$ and taking the supremum norm directly yields the standard estimate on the 
interpolation error in the $L^{\infty}$ 
norm, namely
\begin{align}\label{eq:power_function_bound}
\Vert f - \Pi_{V(X_N)}(f) \Vert_{L^\infty(\Omega)} \leq \Vert P_N \Vert_{L^\infty(\Omega)} \cdot \Vert f \Vert_{\ns}, ~ f \in \ns.
\end{align}

Instead of expressing the interpolant in terms of the basis of so called kernel translates $k(\cdot, x_i)$, it is sometimes useful to consider the Lagrange 
basis. The 
following proposition collects some classical results in this direction (see e.g. \cite{Wendland2005}).
\begin{prop}\label{prop:lagrange_interpolant}
Let $X_N\subset\Omega$ be pairwise distinct points. Then there exists a unique Lagrange basis $\left\{l_j\right\}_{j=1}^N$ of $V(X_N)$, i.e., $V(X_N) = 
\Sp\{l_j, 
1\leq j\leq N\}$ and $l_j(x_i) = \delta_{ij}$ for all $1\leq i, j\leq N$. 
Furthermore, for any $f\in\calh$ the unique interpolant of $f$ on $X_N$ can be expressed as 
\begin{align*}
\Pi_{V(X_N)}(f)(x) 
= \sum_{j=1}^N f(x_j) l_j(x)\;\;\fa x\in \Omega,
\end{align*}
and the square of the Power function is given by 
\begin{align*}
P_N(x)^2 = k(x, x) - \sum_{j=1}^N k(x, x_j) l_j(x) \;\;\fa x\in\Omega. 
\end{align*}
Moreover, the Lebesgue constant
\begin{align*}
\Lambda_{X_N}:= \max\limits_{x\in\Omega} \sum_{j=1}^N \left|l_j(x) \right| 
\end{align*}
gives an upper bound on the sampling stability of the interpolation process, i.e.
\begin{align*}
\norm{L^{\infty}(\Omega)}{\Pi_{V(X_N)}(f)}\leq \Lambda_{X_N} \norm{\infty}{f|_X} \;\;\fa f\in\calh. 
\end{align*}
\end{prop}

As mentioned in the introduction, the success of this interpolation process depends crucially on the distribution of the points $X_N$ inside $\Omega$, which we 
will 
quantify by means of the fill distance and the separation distance defined in \eqref{eq:fill_and_sep_dist}. 
Using simple geometric arguments, which relate the volume of the domain $\Omega$ or a surrounding ball 
to the sum of the volumes of balls around the points $x_i \in X_N$, 
one can conclude some basic estimates on the fill distance and on the separation distance (see e.g. \cite{Mueller2009}).
\begin{theorem} \label{th:bounds_fill_sep_dist}
Let $\Omega \subset \mathbb{R}^d$ be bounded, $(X_N)_{N \in \mathbb{N}}$ be a sequence of sets of points within $\Omega$. Then there are constants $c_\Omega > 
0, 
C_\Omega' 
> 0$ such that
\begin{equation}
\begin{aligned}
h_N &\geq c_\Omega \cdot N^{-1/d}, \\
q_N &\leq C_\Omega' \cdot N^{-1/d}.
\end{aligned}
\end{equation}
\end{theorem}
We recall that whenever for a given sequence of point sets there exists $C > 0$ such that additionally $h_N \leq C \cdot q_N$ holds for all $N \in \mathbb{N}$, 
then $(X_N)_{N \in \mathbb{N}}$ is said to be \textit{asymptotically uniformly distributed}.

For the selection of these sampling points we are interested in greedy algorithms which, as mentioned, construct a \textit{nested} sequence of points $(X_N)_{N \in \mathbb{N}}$ starting from  $X_0 := \emptyset$ and updating it at each iteration as $X_{N+1} := X_{N} \cup \{ x_{N+1} \}$ by selecting a point that maximizes a given selection criterion $\eta^{(N)}$ over $\Omega$. In particular, the commonly used $f$-, $P$-, and $f/P$-greedy selection rules use the 
following error indicators:
\begin{enumerate}[label=\roman*.]
\item $f$-greedy: \hspace{7.75mm} $\eta_f^{(N)}(x) = |f(x) - \Pi_{V(X_N)}(f)(x)|$
\item $P$-greedy: \hspace{7mm} $\eta_P^{(N)}(x) = P_{X_N}(x)$
\item $f/P$-greedy: \hspace{3.4mm} $\eta_{f/P}^{(N)}(x) = |f(x) - \Pi_{V(X_N)}(f)(x)|/P_{X_N}(x)$.
\end{enumerate}
The first two criteria are clearly aiming at reducing the pointwise interpolation error either directly, or via \eqref{eq:power_function_bound} by reducing the Power function. The $f/P$-greedy ("$f$ over $P$ greedy") selection, instead, combines the two and it can be proven to be locally optimal, i.e., it provides the best possible reduction of the interpolation error in the native space norm, at each iteration. However, for $f/P$-greedy it is necessary to perform the maximization over $\Omega \setminus X_N$ as the fraction is not well-defined in already selected points $X_N$.

We remark that here and in the following if the maximum in a selection rule is not unique then any of the points realizing the maximal value can be chosen arbitrarily. Whenever we will prove or recall results about the distribution of \textit{the} sequence of points selected by a greedy algorithm we refer to \textit{any} possible sequence.

So far our discussion applies to any given strictly positive definite kernel, but most of the analytical results of this paper are specialized to the 
remarkable case of 
kernels generating Sobolev spaces. To be more precise, we consider the class of translational invariant kernels, i.e., there exists a function $\Phi: 
\mathbb{R}^d \rightarrow \mathbb{R}$ such that the kernel can be 
expressed as $k(x,y) = \Phi(x - y)$. A special case is given by radial basis function kernels which can be expressed as $k(x,y) = \Phi(\Vert x-y \Vert_2)$  
with a \textit{radial basis function} $\Phi: \mathbb{R}^d \rightarrow \mathbb{R}$. Depending on the Fourier 
transform of this function $\Phi$, the native space $\ns$ can be characterized in terms of Sobolev spaces. That means that if there exist constants $c_\Phi, 
C_\Phi > 0$ and $\tau > d/2$ such that
\begin{align} \label{eq:asymptotic_fourier_transform}
c_\Phi (1+\Vert \omega \Vert_2^2 )^{-\tau} \leq \hat{\Phi}(\omega) \leq C_\Phi (1+\Vert \omega \Vert_2^2)^{-\tau} ~~ \forall \omega \in \mathbb{R}^d,
\end{align}
then the native space $\mathcal{H}_k(\mathbb{R}^d)$ is norm equivalent to the Sobolev space $W_2^\tau(\mathbb{R}^d)$. Under some mild conditions on the 
boundary this 
result also holds for domains $\Omega \subset \mathbb{R}^d$, which will be assumed in the following. Kernels whose native space is norm equivalent to such a 
Sobolev space 
$W_2^\tau$ will be called \textit{kernels of finite smoothness} $\tau$.

Generally, on the Sobolev spaces $W_p^k(\Omega)$, we use the following notation to denote the usual (semi-)norms for $k \in \mathbb{N}$,
\begin{align*}
|u|_{W_p^k(\Omega)} &:= \left( \sum_{|\alpha| = k} \Vert D^\alpha u \Vert_{L^p(\Omega)}^p \right)^{1/p} ~~ \text{for} ~ 1 \leq p < \infty, \\
|u|_{W_\infty^k(\Omega)} &:= \sup_{|\alpha| = k} \Vert D^\alpha u \Vert_{L^\infty(\Omega)},
\end{align*}
where $\alpha\in \N^d$ is a multiindex and $|\alpha|:=\alpha_1+\dots+\alpha_d$ is its length.

Moreover, for this class of kernels error estimates are available. A very general way to introduce them is provided by the 
following Theorem \ref{th:estimate_derivatives_wendland}. 
A first version of this theorem was first proved in \cite[Theorem 2.12]{Narcowich2004} and \cite[Theorem 2.6]{WendlandRieger2005} with slightly 
different assumptions on the indices. The 
improvements on these indices were justified in \cite{Narcowich2006}. The final form stated here in Theorem \ref{th:estimate_derivatives_wendland} was taken 
from 
\cite[Theorem 2.2]{LEGIA2006124}, with notation adapted to this paper.

\begin{theorem}
\label{th:estimate_derivatives_wendland}
Suppose $\Omega \subset \mathbb{R}^d$ is a bounded domain satisfying an interior cone condition and having a Lipschitz boundary. Let $X \subset \Omega$ be a 
discrete set 
with sufficiently small fill distance $h = h_{X,\Omega}$. Let $\tau = k + s$ with $k \in \mathbb{N}, 0 \leq s < 1, 1 \leq p < \infty, 1 \leq q \leq \infty, m 
\in 
\mathbb{N}_0$ with $k>m+n/p$ if $p>1$ or $k \geq m + n/p$ if $p=1$. Then for each $u \in W_p^\tau(\Omega)$ we have that
\begin{align*}
|u|_{W_q^m(\Omega)} \leq C \left( h^{\left(\tau-m-d(1/p-1/q)_+\right)} \cdot |u|_{W_p^{\tau}(\Omega)} + h^{-m} \Vert u|_X \Vert_\infty \right)
\end{align*}
where $C>0$ is a constant independent of $u$ and $h$ and $(x)_+ = \max\{x,0\}$.
\end{theorem}

Replacing $u$ with the residual $f - \Pi_{V(X_N)}$ makes the second term in the right hand side vanish, and this idea is used to obtain error bounds for the 
error of the interpolation of $f$ and its derivatives, w.r.t. a suitable $L^q$ norm.

We recall that for these kernels of finite smoothness, several convergence results also for the three greedy algorithms are available. In particular, 
quasi-optimal rates 
have been proven in \cite{SH16b} for the $P$-greedy algorithm, while the results for $f$-greedy (see \cite{Mueller2009}) and for $f/P$-greedy (see 
\cite{Wirtz2013}) 
are optimal only under some restrictive assumptions, and when these are not satisfied there is a significant gap between these proved rates and the one 
observed in numerical experiments. As 
mentioned in the introduction, one of the goals of this paper is to show that under small modifications, also these algorithms can be proven to have worst-case 
optimal 
convergence rates.

\begin{rem}
Although this paper is focused on the interpolation of scalar-valued functions, we would like to mention that most of the results can be extended to deal with 
vector-valued functions. Indeed, it is possible to define a notion of strictly positive definite matrix-valued kernel (see \cite{Micchelli2005}), which allows 
to 
generalize the theory of kernels and native spaces to the case of vectorial functions. Among a vast variety of matrix-valued kernels, a particularly simple 
class is 
given 
by the so-called separable kernels, which are obtained by linear combinations of scalar kernels with positive definite matrices \cite{Wittwar2018}. 

In particular, a very effective and common choice to approximate functions with values in $\R^q$ for some $q\geq 1$ is to consider a matrix valued kernel $K(x, 
y):= k(x, y) I$, where $I$ is the $q\times q$ identity matrix and $k$ is a standard scalar-valued kernel. This choice is one of the fundamental tools of 
the Vectorial Kernel Orthogonal Greedy 
Algorithm (VKOGA) of \cite{Wirtz2013}, which extends the greedy algorithms to vectorial functions by selecting a set of points which is shared over the $q$ 
components, 
thus resulting in an interpolant with fewer centers, which is hence faster to evaluate.

It has been proven in \cite{Wittwar2018} (see Remark 1 and Lemma 3.8 in that paper) that the native space of this matrix-valued kernel $K$ is given by the 
tensor product 
of $q$ copies of the standard native space of the scalar-valued kernel $k$, and in particular that the Power function of $K$ is defined as a vector-valued 
function whose 
components are identical and equal to the usual Power function of $k$. 

Thus, the upper and lower bounds on the scalar-valued Power function and the related results obtained in this paper can immediately be translated to the 
vectorial case by component-wise application, 
if this matrix valued kernel $K:=k I$ is used.
\end{rem}

\section{Stabilized greedy algorithms} \label{sec:restricted_greedy_alg}

As a motivation for the introduction of stabilized greedy algorithms one can observe that the $f/P$-greedy algorithm does not need to be well defined. The 
reason is that there might not exist a maximum of $|f-\Pi_{V(X_N)}(f)|/P_N$ within $\Omega \setminus X_N$. This set is open, so the supremum of 
$|f-\Pi_{V(X_N)}(f)|/P_N$ needs not to be attained. This situation is presented in the following example.

\begin{example}
\label{ex:ill_definedness_fP}
Consider the interval $[0,1]$ and the unscaled linear Mat\'ern kernel $k(x,y) = (1+|x-y|) \cdot \exp(-|x-y|)$, and take 
\begin{align*}
f(x) := -x + x^2 + k(x,0) \stackrel{x \geq 0}{=} -x + x^2 + (1+x) \cdot \exp(-x).
\end{align*}
The maximum of $f$ is attained in $0$, thus the residual $r_1$ is given by
\begin{align*}
r_1(x):= f(x) - s_1(x) = f(x) - f(0) \cdot k(x,0) = -x + x^2.
\end{align*}
The Power function $P_1$ is given by $P_1(x) = \sqrt{1-(1+x)^2 \cdot \exp(-2x)}$. The calculation of $\lim_{x \searrow 0} |r_1|^2/P_1^2$ with l'H{\^o}pital's 
rule 
shows that $\lim_{x \searrow 0} \frac{r_1(x)^2}{P_1(x)^2} = 1$. Furthermore for $0<x\leq 1$ it can be estimated  that $r_1(x)^2/P_1(x)^2 < 1$. Thus $|r_1|/P_1$ 
approaches its supremum for $x \rightarrow 0$, but $0 \notin \Omega \setminus X_1$ since $X_1 = \{ x_1 \} = \{ 0 \}$. Thus the $f/P$ greedy procedure is not 
well defined. 

A plot of the function $f$, the residual $r_1$, the Power functions $P_0, P_1$ and the ratio $|r_1|/P_1$ as well as the first interpolation point $x_1$ is 
given in 
Figure \ref{fig:motivating_example_2}.

\begin{figure}[ht] 		%
\newlength\fwidth
\setlength\fwidth{0.675\textwidth}	%
\centering
\input{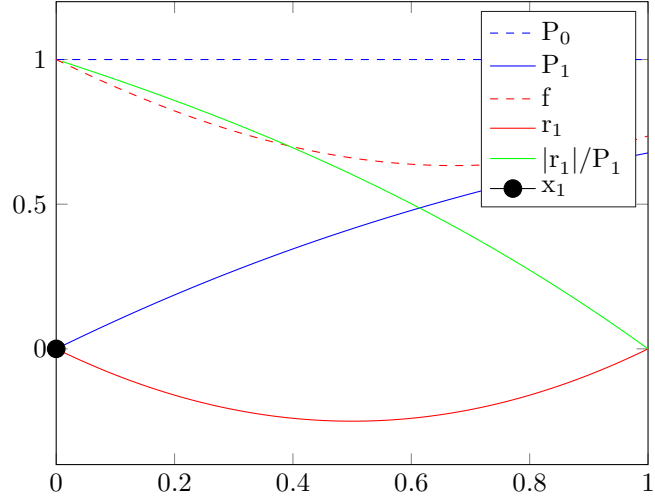}
\caption[Motivating example.]{Motivating example for the $\gamma$-stabilized algorithm based on the linear Mat\'ern kernel: The maximum of $|r_1|/P_1$ is 
attained in $x_1$, which is not element of $\Omega \setminus X_1$. The horizontal axis describes the domain $\Omega = [0,1]$ whereas the vertical axis 
describes the function values.} \label{fig:motivating_example_2}
\end{figure}

\end{example} 

In numerical implementations the point selection is performed on a discretization of the set $\Omega$, such that the algorithm will still work. Nevertheless it 
might still happen that the next point is chosen close to an old interpolation point such that the separation distance drops significantly, which worsens the 
stability of the interpolation procedure. \\ %
One way to circumvent this limitation is given by the \textit{$\gamma$-stabilized} or \textit{$\gamma$-restricted} greedy algorithms, which will be introduced 
in the following definition. 
The notation \textit{restricted} will be clear after the definition, while the notation \textit{stabilized} will be explained in Section \ref{sec:uniformity} 
based on derived stability results. 

\begin{definition}
Let $X_0:=\emptyset$ and $0<\gamma \leq 1$. For every $N\in\N\cup\{0\}$, define $\Omega_\gamma^{(N)} := \{ x \in \Omega: P_{X_N}(x) \geq \gamma \cdot \Vert 
P_{X_N} 
\Vert_\infty \}$. 

Then the $\gamma$-stabilized or $\gamma$-restricted greedy selection criterion is defined as 
\begin{align*}
x_{N+1} = \argmax_{x \in \Omega_\gamma^{(N)}} \eta^{(N)}(x),
\end{align*}
where $\eta^{(N)}: \Omega \rightarrow \mathbb{R}_{\geq 0}$ is some given selection criterion.
\end{definition}

The naming \textit{restricted} of the selection criterion is clearly based on the restriction of the admissible set of points from the set $\Omega$ to $\Omega_\gamma^{(N)}$. 
Moreover in the following, when we want to focus on the restricted set $\Omega_\gamma^{(N)}$ of admissible points while the concrete error indicator 
$\eta^{(N)}$ is not of decisive importance, we will denote this as \textit{any $\gamma$-stabilized algorithm}.

We stress that the definition can be applied to any already existing selection criterion. This yields for example the $\gamma$-stabilized $f/P$-greedy 
algorithm, 
which selects the next 
points according to 
\begin{align*}
x_{N+1} = \argmax_{x \in \Omega_\gamma^{(N)}} \frac{|(f-\Pi_{V(X_N)}(f))(x)|}{P_N(x)}.
\end{align*}

Moreover, for $\gamma=0$ it holds $\Omega_\gamma^{(N)} = \Omega$, thus there is no restriction at all and the algorithm coincides with the standard greedy 
algorithm with selection criterion $\eta^{(N)}$. For example, the $\gamma$-stabilized $f/P$-greedy reduces in this 
case to the usual $f/P$-greedy algorithm. 
That is the reason why the case $\gamma=0$ is excluded from the definition. 

For $\gamma=1$, instead, it holds $\Omega_1^{(N)} = \{ x \in \Omega: P_{X_N}(x) = \Vert P_{X_N} \Vert_\infty \}$, i.e., the selection 
criterion reduces to the usual $P$-greedy criterion.

All together, these stabilized algorithms provide all the intermediate cases between $\gamma=0$ and $\gamma=1$, i.e., they interpolate between standard, known 
methods to 
produce a new scale of selection criteria with a potentially significantly increased flexibility.

\section{Convergence rate of the Power function} \label{sec:conv_rates}

In this section we give two results. First we prove an upper bound on the decay of the maximal value of the Power function by generalizing results of 
\cite{SH16b}. This, using the bound \eqref{eq:power_function_bound}, immediately leads to convergence of the interpolation procedure with the help of any 
$\gamma$-stabilized greedy algorithm, independently of the chosen selection rule. Then we prove a lower bound on the decay of the maximal value of the Power 
function, which implies it cannot drop arbitrarily fast. Since the lower and upper bounds coincide up to constants, we can conclude that they are quasi-optimal.

Note that in the next theorem, and in most of the results of this paper, the bounds contain constants that depend on negative powers of $\gamma$, i.e., 
a smaller restriction parameter yields a larger constant. For the limit $\gamma \rightarrow 0$ this factor tends to infinity, thus no statements on the 
unstabilized algorithms are possible. In all other cases the restriction parameter is positive and a fixed parameter of the problem.

The proof of the following Theorem \ref{th:decay_power_func_restricted} is along the lines of the analysis within \cite{SH16b}, with some modifications 
required to include the effect of the parameter $\gamma$. 

\begin{theorem} \label{th:decay_power_func_restricted}
Assume that $\Omega \subset \mathbb{R}^d$ is a compact domain which satisfies an interior cone condition and has a Lipschitz boundary. Suppose that $k$ is a 
translational invariant kernel of finite smoothness $\tau > d/2$. Then any $\gamma$-stabilized algorithm applied to a function in $f\in\ns$ gives a sequence of 
point sets $X_N \subset \Omega$ such that 
\begin{align}
\label{eq:upper_bound_power_function}
\Vert P_{N} \Vert_{L^\infty(\Omega)} \leq C_P \cdot \gamma^{-2} \cdot N^{\frac{1}{2}-\frac{\tau}{d}}
\end{align}
holds with $C_P = 2^{5\tau/d-1/2} \cdot c_1$. The constant $c_1$ is independent of $\gamma$ and $N$. In particular, it holds
\begin{align}\label{eq:gamma_error_bound}
\Vert f - \Pi_{V(X_N)}(f) \Vert_{L^\infty(\Omega)}
\leq C_P \cdot \gamma^{-2} \cdot N^{\frac{1}{2}-\frac{\tau}{d}} \Vert f \Vert_{\ns}.
\end{align}
\end{theorem}

\begin{proof}
To prove the decay rate, one does not need the concrete selection criterion, it is only necessary to focus on the constraint given by the $\gamma$-restriction 
to $\Omega_\gamma^{(N)} \equiv \{ x \in \Omega ~ | ~ P_N(x) \geq \gamma \cdot \Vert P_N \Vert_\infty \}$, i.e., the case where in any step a point is chosen 
arbitrarily from the restricted set. %
this case includes the proof for the decay rates of the $\gamma$-stabilized algorithms. \newline
Since the proof is based on a Corollary from \cite{DeVore2013} and a linkage which was used in \cite{SH16b}, an overview will be given in the following:
\begin{enumerate}[label=\roman*.]

\item To begin, we describe the setting within \cite{DeVore2013} which will be called \textit{abstract setting}. Although the cited paper deals with general 
Banach 
spaces, we report the results already adapted to the Hilbert space setting.

Consider a Hilbert space $(\mathcal{H}, \Vert \cdot \Vert)$ and a compact subset $\mathcal{F} \subset \mathcal{H}$ which is assumed to satisfy $\Vert f \Vert 
\leq 1 ~ 
\forall f \in \mathcal{F}$. For elements $f_1, .., f_N \in \mathcal{F}$ define $V_N := \text{span}\{ f_1, .., f_N\} \subset \mathcal{H}$. The question is how 
well does 
the subspace $V_N$ approximate $\mathcal{F}$. For this reason the following two definitions are introduced:
\begin{align*}
\sigma_N :=& \sigma_N(\mathcal{F})_\mathcal{H} := \sup_{f \in \mathcal{F}} \Vert f-\Pi_{V_N}(f) \Vert_\mathcal{H}, \\
d_N :=& d_N(\mathcal{F})_\mathcal{H} := \inf_{\substack{Y \subset \mathcal{H} \\ \text{dim}(Y) = N}} \sup_{f \in \mathcal{F}} \Vert f-\Pi_{Y}(f) 
\Vert_\mathcal{H}.
\end{align*}
In the literature, $d_N(\mathcal{F}_\mathcal{H})$ is usually called the \textit{Kolmogorov N-width}. For $N=0$ we set $Y = \{ 0 \}$ respective $V_N = \{ 0 \}$. 
Moreover, the subscript $\mathcal{H}$ will be mostly dropped, and the quantity $\sigma_N$ is further defined for single elements $f \in \mathcal{H}$ via 
$\sigma_N(f) := 
\sigma_N(\{f\})$. 

The following algorithm, which was called weak greedy algorithm with constant $\gamma$ in a Hilbert space $\mathcal{H}$, was investigated in \cite{DeVore2013}:
\begin{itemize}
\item First step: Choose $f_1 \in \mathcal{F}$ such that $\sigma_0(f_1) = \Vert f_1 \Vert \stackrel{!}{\geq} \gamma \cdot \sigma_0(\mathcal{F})_\mathcal{H}$.
\item Iterative step: Given $\{f_1, .., f_N\}$ already chosen, and defining $V_N := \text{span}\{f_1, .., f_N\}$, 
choose $f_{N+1} \in \mathcal{F}$ such that 
$$\sigma_N(f_{N+1}) = \Vert f_{N+1} - \Pi_{V_N}(f_{N+1}) \Vert \stackrel{!}{\geq} \gamma \cdot \sigma_N(\mathcal{F}).$$
\end{itemize}
The following statement is a result within \cite{DeVore2013}, namely Corollary 3.3, ii). It will be used for this proof:
\begin{itemize}
\item[($\star$)] For the weak greedy algorithm with constant $\gamma$ in a Hilbert space $\mathcal{H}$ we have the following: 
\begin{itemize}
\item[ii)] If $d_N(\mathcal{F}) \leq C_0 \cdot N^{-\alpha} ~  \forall N \in \mathbb{N}$, then it holds $\sigma_N(\mathcal{F}) \leq C_1 \cdot N^{-\alpha} ~ 
\forall N \in \mathbb{N}$ with $C_1 = 2^{5\alpha+1} \cdot \gamma^{-2} \cdot C_0$.
\end{itemize}
\end{itemize}

\item The paper \cite{SH16b} used this previously stated result and applied it to the setting of kernel interpolation, which will be called \textit{kernel 
setting} in 
the following.

For this purpose choose $\mathcal{H} = \ns$ and $\mathcal{F} = \{ k(\cdot, x): x \in \Omega\}$. Without loss of generality, we assume $\Vert k(\cdot, x) \Vert_{\ns} \leq 1$, otherwise consider the kernel normalized with $\Vert k(\cdot, x) \Vert_{\ns} = k(x,x)^{1/2} \equiv \text{const}$. The fact that the choice of $\mathcal{F}$ satisfies the 
requirements can be seen in Lemma 3.1.\ of \cite{SH16b}. The choice of $\mathcal{F} = \{ k(\cdot, x): x \in \Omega \}$ means that any $f = k(\cdot, x) \in \mathcal{F}$ can be uniquely associated with an $x \in \Omega$ and vice versa. This is the key ingredient to link the abstract setting with the kernel setting. Thus it holds $V_N \equiv 
\text{span}\{f_1, .., f_N\} = \text{span}\{k(\cdot, x_i), x_i \in X_N\}$ where $X_N$ is the set of points which correspond to the functions $f_1, .., f_N$ of the abstract setting. Furthermore for $f \in \mathcal{F}$ we get
\begin{align}
\sigma_N(f) \equiv& \Vert f - \Pi_{V_N}(f) \Vert_{\ns} = \Vert k(\cdot, x) - \Pi_{V_N}(k(\cdot, x)) \Vert_{\ns} \nonumber \\
=& P_N(x) \label{eq:linkage1}, \\
\sigma_N(\mathcal{F}_h) \equiv& \sup_{f \in \mathcal{F}_h} \Vert f-\Pi_{V_N}(f) \Vert_{\ns} = \sup_{x \in \Omega} \Vert k(\cdot, x) - \Pi_{V_N}(k(\cdot, x)) 
\Vert_{\ns} \nonumber \\ 
=& \Vert P_N \Vert_{L^\infty(\Omega)} \label{eq:linkage2},
\end{align}
where the representation of the Power function from Equation \eqref{eq:power_function_as_norm} was used.

\item Now we turn to the current algorithm:
\begin{itemize}
\item First step: Choose $f_1 = k(\cdot, x_1) \in \mathcal{F} \equiv \{ k(\cdot, x): x \in \Omega \}$ such that 
$\Vert f_1 \Vert = \Vert k(\cdot, x_1) \Vert 
\geq \gamma \cdot \sigma_0(\mathcal{F}) = \gamma \cdot \Vert P_0 \Vert_{L^\infty(\Omega)}$.
\item Iterative step: $\{f_1, .., f_N\} = \{k(\cdot, x_1), .., k(\cdot, x_N)\} \subset \mathcal{F}$ is already chosen, and $V_N = \text{span}\{f_1, .., f_N\} = 
\text{span}\{k(\cdot, x_1), .., k(\cdot, x_N)\}$. This means that we have already chosen points $\{x_1, .., x_N\} =: X_N$. Now choose $f_{N+1} \in \mathcal{F}$ (associated with $x_{N+1}$) such that
\begin{align*}
\sigma_N(f_{N+1}) \geq& ~ \gamma \cdot \max_{f \in \mathcal{F}} \sigma_N(f) = \gamma \cdot \sigma_N(\mathcal{F}), \\
\stackrel{\eqref{eq:linkage1}}{\Leftrightarrow} P_N(x_{N+1}) \geq& ~ \gamma \cdot \max_{x \in \Omega} P_N(x) = \Vert P_N \Vert_{L^\infty(\Omega)},
\end{align*}
where Equation \eqref{eq:linkage1} was used. Thus weak greedy algorithms with constant $\gamma$ in the Hilbert space $\ns$ exactly correspond to $\gamma$--stabilized greedy kernel algorithms which select any point within $\Omega_\gamma^{(N)}$.
\end{itemize}
Thus the results from the abstract setting can be used, especially the Corollary $(\star)$ which was stated in the end of the first point. The requirements on 
$d_N(\mathcal{F})$ are satisfied due to Lemma 3.1.\ in \cite{SH16b}, which states $d_N(\mathcal{F}) \leq c_1 \cdot N^{-\alpha_1}$. Thus the same decay can be 
concluded for any $\gamma$-stabilized greedy algorithm. In detail, using the stated decay rates for $d_N(\mathcal{F})$ from Lemma 3.1. in \cite{SH16b} it holds 
for kernels $k$ with finite smoothness $\tau > d/2$:
\begin{align*}
\Vert P_N \Vert_{L^\infty(\Omega)} \leq C_P \cdot \gamma^{-2} \cdot N^{\frac{1}{2}-\frac{\tau}{d}}
\end{align*}
with $C_P = 2^{5\tau/d-1/2} \cdot c_1$ (since $\alpha = \tau/d - 1/2$).
\end{enumerate}
Thus Inequality \eqref{eq:upper_bound_power_function} is proven and an application of \eqref{eq:power_function_bound} yields the bound 
\eqref{eq:gamma_error_bound}.
\end{proof}

\begin{rem}
Observe that Theorem \ref{th:decay_power_func_restricted} is formulated in terms of a fixed function $f\in\ns$. This is done since a general selection rule 
$\eta^{(N)}$ 
may be function dependent, and thus a target function needs to be specified. Nevertheless, the error bound \eqref{eq:gamma_error_bound} is based on the Power 
function 
only, and thus it applies also to any $g\in\calh$ with $g\neq f$.

Moreover, if a function independent selection rule is used, there is no need to specify target values to run the corresponding $\gamma$-stabilized version of 
the 
algorithm.
\end{rem}

Theorem \ref{th:decay_power_func_restricted} is the only result of this paper that applies not only to the class of kernels of finite smoothness, and it can 
indeed be extended to any kernel for which there exists known algebraic or exponential convergence results for fixed point distributions, which can be 
transfered to the ones of the $\gamma$-greedy Power function. We refer to \cite{SH16b} for more details on this idea, and we mention here only the following 
additional result, which covers the case of the Gaussian kernel.
\begin{theorem}
Assume that $\Omega \subset \mathbb{R}^d$ is a compact domain which satisfies an interior cone condition and has a Lipschitz boundary. Suppose that $k$ is a 
translational invariant kernel of infinite smoothness like the Gaussian kernel. Then any $\gamma$-stabilized algorithm applied to a function in $f\in\ns$ gives 
a sequence of point sets $X_N 
\subset \Omega$ such that 
\begin{align*}
\Vert P_{N} \Vert_{L^\infty(\Omega)} \leq c_2 \cdot \gamma^{-1} \cdot e^{-c_3 \cdot N^{-1/d}}
\end{align*}
holds. The constants $c_2$ and $c_3$ are independent of $\gamma$ and $N$. In particular, we obtain
\begin{align}
\Vert f - \Pi_{V(X_N)}(f) \Vert_{L^\infty(\Omega)}
\leq c_2 \cdot \gamma^{-1} \cdot e^{-c_3 \cdot N^{-1/d}} \Vert f \Vert_{\ns}.
\end{align}
\end{theorem}

The next Theorem \ref{th:lower_bound_power_func} gives a lower bound on the decay of the maximal value of the Power function. The proof requires the following 
preliminary Lemma \ref{lemma:SchabackDeMarchiEstimate}, which estimates the native space norm of a bump function. It is part of a proof within \cite[Theorem 
1]{DeMarchi2010}, where we corrected a minor mistake in an exponent. 

\begin{lemma}
\label{lemma:SchabackDeMarchiEstimate}
Let $\Psi \in C^\infty(\R^d)$ be a bump function, i.e., having support within the unit ball and satisfying $\Psi(0)=1, \Vert \Psi \Vert_{L^\infty(\R^d)} = 1$. 
Consider 
a bounded domain $\Omega \subset \mathbb{R}^d$ satisfying an interior cone condition and a native space which is norm equivalent to $W_2^\tau(\Omega)$ with 
$\tau > 
\frac{d}{2}$. Then the native space norm of a rescaled version of $\Psi$ with $0 < p \leq 1$ can be estimated as 
\begin{align*}
\left \Vert \Psi \left( \frac{\cdot}{p} \right) \right \Vert_{\ns} \leq C \cdot p^{\frac{d}{2}-\tau} \cdot \Vert \Psi \Vert_{\ns}.
\end{align*}
\end{lemma}

Now the following Theorem \ref{th:lower_bound_power_func} states a lower bound on the decay of the Power function. The result holds in general for any 
distribution of the 
points and hence especially for the point distribution created by any greedy algorithm.

\begin{theorem} \label{th:lower_bound_power_func}
Assume that $k$ is a translational invariant kernel with finite smoothness $\tau > d/2$. Let $\Omega \subset \mathbb{R}^d$ be bounded, $(X_N)_{N \in 
\mathbb{N}}$ 
be a sequence of sets of points within $\Omega$. Then there exists a constant $c_P > 0$ such that
\begin{align*}
\Vert P_N \Vert_{L^\infty(\Omega)} \geq c_P \cdot N^{\frac{1}{2}-\frac{\tau}{d}}.
\end{align*}
\end{theorem}

\begin{proof}
The setting of Lemma \ref{lemma:SchabackDeMarchiEstimate} with $p := h_N$ is considered. Thus the following inequality on the native space norm of a bump 
function rescaled and shifted to some point $\tilde{x} \in \Omega$  holds:
\begin{align}
\label{eq:estimate_within_proof_lower_bound}
\frac{1}{\left \Vert \Psi \left( \frac{\cdot - \tilde{x}}{h_N} \right) \right \Vert_{\ns}} \geq \frac{1}{C} \cdot \frac{h_N^{\tau - d/2}}{\Vert \Psi 
\Vert_{\ns}}.
\end{align}

Since the bump function is smooth and the native space is norm equivalent to a Sobolev space, the bump function and any rescaled and shifted version of it are 
in the native space. Moreover, if we fix a set $X_N\subset\Omega$ of points, since the rescaled function $\Psi_{h_N, \tilde{x}} := \Psi((\cdot - \tilde{x}) / 
h_N)$ has only support within a ball of radius $h_N$, it can be placed among the sampling points such that no sampling point lies in the interior of the 
support of $\Psi_{h_N, \tilde{x}}$ \footnote{In fact it is needed to rescale with $(1+\epsilon)h_N$ with $\epsilon>0$ fixed, e.g. $\epsilon = 0.01$. The reason is that the supremum within the definition of the fill distance needs not to be attained. This means we choose $\tilde{x} \in \Omega$ such that $h_N \equiv \sup_{x \in \Omega} \min_{x_i \in X_N} \Vert x-x_i \Vert_2 \leq (1+\epsilon) \cdot \min_{x_i \in X_N} \Vert \tilde{x} - x_i \Vert_2$. For the sake of simplicity we drop these technical details here.}. 
Thus the interpolant to $\Psi_{h_N, \tilde{x}}$ is the zero function. 
Therewith one can conclude:
\begin{align*}
P_N(x) =& \sup_{0 \neq f \in \ns} \frac{|(f-\Pi_{V(X_N)}(f))(x)|}{\Vert f \Vert_{\ns}} \geq \frac{\Psi_{h_N, \tilde{x}}(x) - \Pi_{V(X_N)}(\Psi_{h_N, 
\tilde{x}})(x)}{\Vert \Psi_{h_N, \tilde{x}} \Vert_{\ns}} \\
=& ~ \frac{\Psi_{h_N, \tilde{x}}(x)}{\Vert \Psi_{h_N, \tilde{x}} \Vert_{\ns}} \\
\Rightarrow \Vert P_N \Vert_\infty =& \sup_{x \in \Omega} \frac{\Psi_{h_N, \tilde{x}}(x)}{\Vert \Psi_{h_N, \tilde{x}} \Vert_{\ns}} = \frac{1}{\Vert \Psi_{h_N, 
\tilde{x}} \Vert_{\ns}} 
\stackrel{(\ref{eq:estimate_within_proof_lower_bound})}{\geq} \frac{1}{C} \cdot \frac{h_N^{\tau - d/2}}{\Vert \Psi \Vert_{\ns}}.
\end{align*}
Using the lower bound on the fill distance $h_N \geq c_\Omega \cdot N^{-1/d}$ from Theorem \ref{th:bounds_fill_sep_dist} and remembering that $\tau > d/2 
\Leftrightarrow 
\tau - d/2 > 0$ finally yields 
\begin{align*}
\Vert P_N \Vert_\infty \geq& ~ \frac{1}{C \cdot \Vert \Psi \Vert_{\ns}} \cdot c_\Omega^{\tau-d/2} \cdot \left( N^{-1/d} \right)^{\tau - d/2} = ~ 
\frac{c_\Omega^{\tau - d/2}}{C \cdot \Vert \Psi \Vert_{\ns}} \cdot N^{\frac{1}{2}-\frac{\tau}{d}}.
\end{align*}
Setting $c_P := c_\Omega^{\tau-d/2} \cdot C^{-1} \cdot \Vert \Psi \Vert_{\ns}^{-1}$ gives the desired result.
\end{proof}

This result can be applied to the case of the $\gamma$-stabilized greedy algorithms, i.e., to the sequence of points obtained by such an algorithm. Thus it 
gives a lower 
bound on the decay of the Power function. If one compares this lower bound with the upper bound within Theorem \ref{th:decay_power_func_restricted}, it is 
clear that the 
decay rates coincide. Thus, as mentioned in the beginning of this section, both results give exact rates and cannot be improved up to constants. Especially, 
this means that the optimal decay rate for the maximal value 
of the Power function can be achieved with greedy methods, for example with the proposed $\gamma$-stabilized greedy method.

\section{Uniformity of the selected points} \label{sec:uniformity}
Based on the previous results some statements about the point distribution of the selected points using kernels of finite smoothness $\tau$ can be derived. The 
first subsection gives an upper estimate on the fill distance, whereas the second subsection derives lower estimates on the separation distance. These results 
prove that the points selected by any $\gamma$-stabilized greedy algorithm are asymptotically uniformly distributed.

\subsection{Estimate on the fill distance}
For a first result on the fill distance we use the following Lemma \ref{lemma:estimate_for_fill_dist}, which is a refinement of Theorem 3.1 in 
\cite{DeMarchi2010}. 

\begin{lemma}
\label{lemma:estimate_for_fill_dist}
Let $\Omega$ be a compact domain in $\mathbb{R}^d$ satisfying an interior cone condition. Suppose that the kernel $k$ is a translational invariant kernel with 
finite smoothness $\tau$. Then there exists a constant $M>0$ with the following property: If $\epsilon > 0$ and $X_N = \{x_1, .., x_N\} \subset \Omega$ are 
given 
such that
\begin{align} \label{eq:epsilon_estimate}
\Vert f - \Pi_{V(X_N)}(f) \Vert_{L^{\infty}} \leq \epsilon \cdot \Vert f \Vert_{\ns} ~ \text{for all} ~ f \in \ns,
\end{align}
then the fill distance of $X_N$ satisfies
\begin{align*}
h_{X_N} \leq M \cdot \epsilon^{1/(\tau - d/2)}.
\end{align*}
\end{lemma}

\begin{proof}
The same argumentation as in the proof of Theorem \ref{th:lower_bound_power_func} is used, i.e., the bump function $\Psi$ from Lemma 
\ref{lemma:SchabackDeMarchiEstimate} 
as well as its rescaled and shifted function $\Psi_{h_N} \equiv \Psi((\cdot - \tilde{x})/ h_N)$ are considered. 
The interpolant $\Pi_{V(X_N)}(\Psi_{h_N})$ to $\Psi_{h_N}$ is again the zero function, so we can conclude
\begin{align*}
1 &= \Vert \Psi_{h_N} \Vert_{L^{\infty}} = \Vert \Psi_{h_N} - \Pi_{V(X_N)}(\Psi_{h_N}) \Vert_\infty \leq \epsilon \cdot \Vert \Psi_{h_N} \Vert_{\ns} \\
&\leq C \cdot \epsilon \cdot h_N^{d/2-\tau} \cdot \Vert \Psi \Vert_{\ns}.
\end{align*}
where in the last line Lemma \ref{lemma:SchabackDeMarchiEstimate} was used. This can be rearranged to conclude the result
\begin{align*}
h_N^{\tau - d/2} &\leq C \cdot \epsilon \cdot \Vert \Psi \Vert_{\ns} \\
\Leftrightarrow ~~~~~~~~~ h_N &\leq M \cdot \epsilon^{1/(\tau-d/2)}
\end{align*}
with $M:=(C \cdot \Vert \Psi \Vert_{\ns})^{1/(\tau - d/2)}$.
\end{proof}

\begin{rem}
We remark that the original proof of Theorem 3.1 in \cite{DeMarchi2005} provided, under the same assumptions, a bound $h_{X_N} \leq M_\alpha \cdot 
\epsilon^{1/(\alpha - d/2)}$ for all $\alpha>\tau$, though not for $\alpha = \tau$.
\end{rem}

Observe that any bound on the Power function yields a way to satisfy the condition stated in Inequality 
\eqref{eq:epsilon_estimate} with $\epsilon = \Vert P_N \Vert_{L^\infty(\Omega)}$. This can be combined with the upper bound on the decay of the Power function 
from Inequality \eqref{eq:upper_bound_power_function} to conclude the following estimate on the fill distance.

\begin{theorem} \label{th:estimate_fill_dist}
Assume that $k$ is a translational invariant kernel with finite smoothness $\tau$. Then there exists a constant $c > 0$ with the following property: Any 
$\gamma$-stabilized algorithm applied to a function on a compact set $\Omega \subset \mathbb{R}^d$ which satisfies an interior cone condition and has a 
Lipschitz boundary gives a sequence of point sets $X_N \subset \Omega$ such that it holds
\begin{align*}
h_N \leq c \cdot \gamma^{-2/(\tau-d/2)} \cdot N^{-1/d}.
\end{align*}
\end{theorem}

This upper bound has a decay of $N^{-1/d}$. Recalling Theorem \ref{th:bounds_fill_sep_dist} which gives a lower bound with the same decay of $N^{-1/d}$ shows 
that these decay rates are exact and cannot be improved further.

\subsection{Estimate on the separation distance}

Based on the lower bound of the decay of the Power function, even statements on the separation distance are possible. 

In the following we prove two different results. First, in Theorem \ref{th:lower_bound_power_func} we obtain an optimal rate of decay of the separation 
distance. Since the proof is mainly based on the application of the mean value theorem, further smoothness is needed, namely $\tau > d/2 + 1$. 

To circumvent this limitation we show another proof strategy in Theorem \ref{th:estimate_dist_Omega}, which gives, however, rates which are optimal only under 
certain 
conditions, which are discussed at the end of this section.

First of all, the following simple Lemma \ref{lem:dist_sep_estimate} is stated. Since the proof easily follows by induction on $N$, we omit it here.
\begin{lemma} \label{lem:dist_sep_estimate}
If there exist constants $c, \alpha > 0$, such that for all $N \in \mathbb{N}$ it holds $\mathrm{dist}(\Omega_\gamma^{(N)},X_N) \geq c \cdot N^{-\alpha}$, then it 
follows $q_{N+1} \geq c \cdot N^{-\alpha} ~ \forall N \in \mathbb{N}$.
\end{lemma}

To prove the first main result Theorem \ref{th:lower_bound_sep_dist2}, an estimate on the derivative of the residual is needed. This is achieved by using 
Theorem \ref{th:estimate_derivatives_wendland}  with $X = X_N$ for $u \in W_p^{\tau}(\Omega)$ satisfying $u|_{X_N} = 0$ and setting, for $\tau > d/2$, the 
values $m=1, q=\infty, s=0, k=\tau, p=2$. Thus the following lemma holds.

\begin{lemma}
\label{lemma:estimate_derivatives_wendland}
Suppose $\Omega \subset \mathbb{R}^d$ is a bounded domain satisfying an interior cone condition and having a Lipschitz boundary. Let $X_N \subset \Omega$ be a 
discrete set with sufficiently small fill distance $h = h_{X_N,\Omega}$. Let $\tau>1+d/2$. Then for each $u \in W_2^\tau(\Omega)$ with $u|_{X_N} = 0$ we have 
that
\begin{align}
\label{eq:estimate_derivatives_wendland}
|u|_{W_\infty^1(\Omega)} \equiv \sup_{|\alpha| = 1} \Vert D^\alpha u \Vert_\infty \leq C h^{\tau-1-d/2} \cdot |u|_{W_2^{\tau}(\Omega)},
\end{align}
where $C>0$ is a constant independent of $u$ and $h$.
\end{lemma}

Moreover, in the proof of Theorem \ref{th:lower_bound_sep_dist2} we need a bound on the derivative of the Power function. This will turn out to involve a 
generalized Power function, which can be used to bound the error in the interpolation of derivatives of $f$, exactly as the standard Power function bounds the 
error in the interpolation of pointwise values of $f$. We recall a construction of this bound in the following theorem. Observe that the condition on the 
functionals is 
in particular realized if $k$ is a translational invariant kernel with smoothness $\tau > d/2+1$ (see e.g. Chapter 16 in \cite{Wendland2005}). 

\begin{lemma}
\label{lemma:power_func_derivative}
Suppose $\Omega\subset\R^d$ and let $k$ be a strictly positive definite kernel. Let $\delta_x\circ \partial_i:\ns\to\R$ be the linear functional given by the 
evaluation 
in $x\in\Omega$ of the the derivative in the $i$-th coordinate direction. 

If $\delta_x\circ \partial_i\in\ns'$ for all $x\in\Omega$ and for all $i=1, .., d$, then it holds that
\begin{equation}
\begin{aligned}
\label{eq:generalized_power_function}
P_N(\delta_x\circ\partial_i) :=& \left \Vert \partial_i^1 k(x, \cdot) - \sum_{j=1}^N \partial_i^1 k(x,x_j) l_j \right \Vert_{\ns} \\
=& \sup_{0 \neq f \in \ns} \frac{| \partial_i f(x) - \partial_i \Pi_{V(X_N)}f(x)|}{\Vert f \Vert_{\ns}},
\end{aligned}
\end{equation}
where $\{l_j\}_{j=1}^N$ denotes the Lagrange basis of Proposition \ref{prop:lagrange_interpolant}.
\end{lemma}

\begin{proof} %
We have that the composition of the point evaluation with the 
partial derivative is a continuous functional, and the corresponding Riesz representer is $v_\lambda = \partial_i^1 k(x, \cdot)$, i.e., it holds $\lambda(f) = 
(\delta_x \circ \partial_i)(f) = (f, v_\lambda)_{\ns} = (\partial_i f)(x)$. Now the proof is quite similar to the proof which shows the equivalence of the two 
expressions in \eqref{eq:generalized_power_function} above for the ordinary Power function, i.e., without the partial derivative.

Recall from Proposition \ref{prop:lagrange_interpolant} that the interpolant of a function $g$ can be expressed in terms of the Lagrange basis as 
$\Pi_{V(X_n)}g(\cdot) = 
\sum_{j=1}^N 
g(x_j) l_j(\cdot)$. Set $g = v_\lambda$ and calculate
\begin{align*}
\lambda(f) - \lambda(\Pi_{V(X_N)}(f)) &= (v_\lambda, (\id-\Pi_{V(X_N)})(f) )_{\ns} \\
&= ((\id-\Pi_{V(X_N)})(v_\lambda), f)_{\ns} \\
&= (v_\lambda - \sum_{j=1}^N v_\lambda(x_j)l_j(\cdot),f)_{\ns}.
\end{align*}
Thus it holds
\begin{align*}
\sup_{0 \neq f \in \ns} \frac{|\partial_i f(x) - \partial_i \Pi_{V(X_N)}f(x)|}{\Vert f \Vert_{\ns}} 
&= \sup_{0 \neq f \in \ns} \frac{|\lambda(f) - \lambda(\Pi_{V(X_N)}(f))|}{\Vert f \Vert_{\ns}} \\
&\hspace{-10mm} = \sup_{0 \neq f \in \ns} \frac{|(v_\lambda - \sum_{j=1}^N v_\lambda(x_j)l_j(\cdot),f)_{\ns}|}{\Vert f 
\Vert_{\ns}} \\
&\hspace{-10mm} \leq \sup_{0 \neq f \in \ns} \frac{\Vert v_\lambda - \sum_{j=1}^N v_\lambda(x_j) l_j(\cdot) \Vert_{\ns} \cdot \Vert f 
\Vert_{\ns}}{\Vert f \Vert_{\ns}} \\
&\hspace{-10mm} = \left \Vert v_\lambda - \sum_{j=1}^N v_\lambda(x_j) l_j(\cdot) \right \Vert_{\ns} \\
&\hspace{-10mm} = \left \Vert \partial_i k(x, \cdot) - \sum_{j=1}^N \partial_i^1 k(x,x_j) l_j(\cdot) \right \Vert_{\ns}.
\end{align*}
Equality is attained by choosing $f = v_\lambda - \sum_{j=1}^N v_\lambda(x_j) l_j(\cdot)$, hence both expressions are equal and the proof is finished.
\end{proof}

Now the first main theorem concerning the separation distance can be formulated. The proof follows the lines of the ones of Lemma 4.2 and Theorem 4.3
 of \cite{DeMarchi2005}, although we include a refined estimate of the derivative of the Power function. Moreover, while the original theorem was used to 
obtain 
a decay of the Power function, we use exactly this decay here to derive a bound on the separation distance.

\begin{theorem} \label{th:lower_bound_sep_dist2}
Let $\Omega \subset \mathbb{R}^d$ be a compact domain satisfying an interior cone condition and having a Lipschitz boundary. Suppose that $k$ is a kernel of 
finite 
smoothness $\tau$ such that $\tau > d/2 + 1$.

Then any $\gamma$-stabilized algorithm yields a sequence of set of points such that
\begin{align*}
q_{N+1} \geq C \cdot \gamma^{4-\frac{2}{\tau-d/2}} \cdot N^{-1/d} ~~ \forall N \in \mathbb{N}.
\end{align*}
\end{theorem}

\begin{proof}%
For structural reasons the proof is split up into two parts:
\begin{enumerate}[label=\roman*.]
\item The representation of the Power function in terms of the Lagrange basis (see Proposition \ref{prop:lagrange_interpolant}) yields a possibility to 
evaluate 
the 
partial 
derivative in the $i$-th direction, $i \in \{1, .., d \}$. Here $\partial_i^1$ denotes the partial derivative in the $i$-th direction with respect to the first 
argument:
\begin{align} 
\partial_i P_N^2(x) =& \partial_i \left( k(x, \cdot) - \sum_{j=1}^N k(x,x_j)l_j, k(x,\cdot) - \sum_{j=1}^N l_j k(x,x_j) \right)_{\ns} \nonumber \\
=& 2 \cdot \left( \partial_i^1 k(x, \cdot) - \sum_{j=1}^N \partial_i^1 k(x,x_j)l_j, k(x,\cdot) - \sum_{j=1}^N l_j k(x,x_j) \right)_{\ns} \nonumber \\
\Rightarrow ~ | \partial_i P_N^2(x) | \leq& 2 \cdot    \underbrace{ \left \Vert \partial_i^1 k(x, \cdot) - \sum_{j=1}^N \partial_i^1 k(x,x_j) l_j \right 
\Vert_{\ns} }_{=: P_N(\delta_x\circ\partial_i)}     \cdot P_N(x) \nonumber \\
=& 2 \cdot P_N(\delta_x\circ\partial_i) \cdot P_N(x). 
\label{eq:lower_bound_sep_dist2_1}
\end{align}
Lemma \ref{lemma:power_func_derivative} can be used together with Equation \eqref{eq:estimate_derivatives_wendland} to estimate $P_N(\delta_x\circ\partial_i)$:
\begin{align*}
P_N(\delta_x\circ\partial_i) &= \sup_{0 \neq f \in \ns} \frac{| \partial_i f(x) - \partial_i \Pi_{V(X_N)}f(x) |}{\Vert f 
\Vert_{\ns}} \\
&\leq \sup_{0 \neq f \in \ns} \frac{\sup_{|\alpha|=1} \Vert D^\alpha (f-\Pi_{V(X_N)}f) \Vert_\infty}{\Vert f \Vert_{\ns}} \\
&\stackrel{\eqref{eq:estimate_derivatives_wendland}}{\leq} \sup_{0 \neq f \in \ns} \frac{C \cdot h_N^{\tau-d/2-1} \cdot 
|f-\Pi_{V(X_N)}f|_{W_2^{\tau}(\Omega)}}{\Vert 
f \Vert_{\ns}}.
\end{align*}
The application of Equation \eqref{eq:estimate_derivatives_wendland} was possible since due to Theorem \ref{th:estimate_fill_dist} the fill distance $h_N$ will 
be sufficiently small for $N$ sufficiently large. Now the semi-norm $| \cdot |_{W_2^{\tau}}$ can be estimated from above by the full norm $\Vert \cdot 
\Vert_{W_2^{\tau}}$ and then the norm equivalence $\Vert \cdot \Vert_{W_2^{\tau}(\Omega)} \asymp \Vert \cdot \Vert_{\ns}$ can be used as well as $\Vert \mathrm{Id} - \Pi_{V(X_N)} \Vert_{\mathcal{L}({\ns})} \leq 1$ due to orthogonality of the projection:
\begin{align*}
P_N(\delta_x\circ\partial_i) &\leq C \cdot h_N^{\tau-d/2-1} \cdot \sup_{0 \neq f \in \ns} \frac{\Vert f-\Pi_{V(X_N)} f 
\Vert_{W_2^{\tau}(\Omega)}}{\Vert 
f 
\Vert_{\ns}} \\
&\leq C' \cdot h_N^{\tau-d/2-1} \cdot \sup_{0 \neq f \in \ns} \frac{\Vert f-\Pi_{V(X_N)} f \Vert_{\ns}}{\Vert f 
\Vert_{\ns}} \\
&\leq C' \cdot h_N^{\tau-d/2-1}.
\end{align*}
This estimate for $P_N(\delta_x\circ\partial_i)$ can be plugged into Inequality \eqref{eq:lower_bound_sep_dist2_1} to obtain
\begin{align}
| \partial_i P_N^2(x) | \leq 2C' \cdot h_N^{\tau-d/2-1} \cdot P_N(x).
\label{eq:lower_bound_sep_dist2_2}
\end{align}
\item Now the mean value theorem can be used to estimate $\text{dist}(\Omega_\gamma^{(N)}, X_N)$. Thus, we consider $\tilde{x} \in \Omega_\gamma^{(N)}$ and 
apply the mean value theorem to $P_N^2$ on the line segment between $\tilde{x}$ and $x_j \in X_N$. This gives a point $\eta = (1-t)\tilde{x} + tx_j$ with $t 
\in [0,1]$. Since it holds $P_N^2(x_j) = 0$, we have
\begin{align}
\label{eq:application_meanvaluetheorem}
\gamma^2 \cdot \Vert P_N \Vert_\infty^2 \leq& P_N^2(\tilde{x}) = (\nabla P_N^2)(\eta) \cdot (\tilde{x}-x_j) \nonumber \\
\leq& \Vert (\nabla P_N^2)(\eta)\Vert_2 \cdot \Vert \tilde{x}-x_j \Vert_2. 
\end{align}
The first factor on the right hand side can be estimated further with help of Equation \eqref{eq:lower_bound_sep_dist2_2} from above:
\begin{align*}
\Vert (\nabla P_N^2)(\eta) \Vert_2 =& \left( \sum_{i=1}^d ((\partial_i P_N^2)(\eta))^2 \right)^{1/2} \\ 
\stackrel{\eqref{eq:lower_bound_sep_dist2_2}}{\leq}& \left( \sum_{i=1}^d (C' \cdot h_N^{\tau-d/2-1} \cdot P_N(\eta) )^2 \right)^{1/2} \\
\leq& C' \cdot h_N^{\tau-d/2-1} \cdot d \cdot \Vert P_N \Vert_{L^{\infty}}.
\end{align*}
Plugging that estimate for $\Vert (\nabla P_N^2)(\eta) \Vert_{2}$ into the Estimate \eqref{eq:application_meanvaluetheorem} yields
\begin{align*}
\gamma^2 \cdot \Vert P_N \Vert_\infty^2 \leq& C' \cdot h_{N}^{\tau-d/2-1} \cdot d \cdot \Vert P_N \Vert_\infty \cdot \Vert \tilde{x}-x_j \Vert_{2} \\
\Leftrightarrow ~~~ \Vert \tilde{x}-x_j \Vert_{2} \geq&
\frac{\gamma^2}{C' \cdot d} \cdot \frac{\Vert P_N \Vert_\infty}{h_N^{\tau-d/2-1}}.
\end{align*}
Now, the upper estimate on the fill distance from Theorem \ref{th:estimate_fill_dist}, namely $h_N \leq c_2 \gamma^{-\frac{2}{\tau-d/2}} N^{-1/d}$ and the 
lower 
estimate 
on the Power function from Theorem \ref{th:lower_bound_power_func}, namely $\Vert P_N \Vert_\infty \geq c_P \cdot N^{1/2-\tau/d}$ can be applied to obtain:
\begin{align*}
\Vert \tilde{x}-x_j \Vert_2 \geq \underbrace{\frac{c_p}{C' \cdot d \cdot c_2^{\tau-d/2-1}}}_{=: \tilde{C}} \cdot 
\frac{\gamma^2}{\gamma^{-\frac{2(\tau-d/2-1)}{\tau-d/2}}} 
\cdot \frac{N^{1/2-\tau/d}}{N^{1/2-\tau/d+1/d}}.
\end{align*}
A simplification of the occurring exponents finally leads to
\begin{align*}
\Vert \tilde{x} - x_j \Vert_2 \geq \tilde{C} \cdot \gamma^{4-\frac{2}{\tau - d/2}} \cdot N^{-1/d}.
\end{align*}
\end{enumerate}
Since $\tilde{x} \in \Omega_\gamma^{(N)}$ was arbitrary, it holds $\text{dist}(\Omega_\gamma^{(N)},X_N) \geq \tilde{C} \cdot \gamma^{4-\frac{2}{\tau - d/2}} 
\cdot 
N^{-1/d}$ and Lemma \ref{lem:dist_sep_estimate} can be applied which directly yields $q_{N+1} \geq \tilde{C} \cdot \gamma^{4-\frac{2}{\tau - d/2}} \cdot 
N^{-1/d}$. \\
This argumentation holds for $N$ sufficiently large. For small $N$ we can simply adjust the constant $\tilde{C}$. Thus taking a proper constant $C$ finishes 
the proof.
\end{proof}
Observe that since $\tau > d/2 + 1$ it holds $2 < 4-2/(\tau-d/2) < 4$, i.e., the exponent of $\gamma$ is positive. \\
Altogether Theorem \ref{th:lower_bound_sep_dist2} states a lower bound on the separation distance which decays as $N^{-1/d}$. Remind that Theorem 
\ref{th:estimate_fill_dist} gives an upper bound on the fill distance which decays also as $N^{-1/d}$. Thus this shows that these decay rates are exact and 
cannot be 
improved further.

\subsubsection{A weaker result under weaker conditions}
Theorem \ref{th:lower_bound_sep_dist2} needed a further restriction on the smoothness of the kernel, namely $\tau > d/2 + 1$. To circumvent this limitation 
another way to 
estimate the distance is proposed in the following, which does not need this assumption.

\begin{theorem} \label{th:estimate_dist_Omega}
Consider a continuous radial basis function kernel $k(x,y) = \Phi(\Vert x-y \Vert)$ with $\Phi(0)=1$ and $\Phi(r)$ monotonically decreasing for $0 \leq r \le
r_0$ for a given $r_0 > 0$. Furthermore $\Phi$ is assumed to satisfy an estimate like $1-\Phi(r) \leq c \cdot r^b$ with $c > 0, b \in \mathbb{N}$ for $0 \leq r 
\leq 
r_0$. 

If any $\gamma$-stabilized greedy algorithm is applied to  a function on a compact set $\Omega \subset \mathbb{R}^d$ which satisfies an interior cone condition 
and has a Lipschitz boundary, then the following asymptotic estimate holds with $C>0$:
\begin{align*}
q_{N+1} \geq C \cdot \gamma^{2/b} \cdot \left( N^{\frac{1}{2}-\frac{\tau}{d}} \right)^{2/b}.
\end{align*}
\end{theorem}

\begin{proof}
The representation of the Power function from Equation \eqref{eq:power_function_as_norm} will be used. Projecting to a smaller subspace yields a worse 
approximation, and in particular using $V(x_j) := \text{span}\{k(\cdot,x_j)\} \subset \text{span}\{k(\cdot, x_i) | x_i \in X_N \} \equiv V(X_N)$ gives the 
following for all $j \in \{1, .., N\}$:
\begin{align*}
P_N(x)^2 \equiv& \Vert k(\cdot, x) - \Pi_{V_N}(k(\cdot, x)) \Vert_{\ns}^2 \\
\leq& \min_{i=1,..,N} \Vert k(\cdot, x) - \Pi_{V(x_i)}(k(\cdot, x)) \Vert_{\ns}^2 \\
=& \min_{i=1,..,N} \Vert k(\cdot, x) - k(x_i,x) \cdot k(\cdot, x_i) \Vert_{\ns}^2
\end{align*}
where $\Vert k(\cdot, x_i)\Vert_{\ns} = k(x_i,x_i)^{1/2} = \Phi(0)^{1/2} = 1$ was used in the last step.
The norm can be expressed via the dot product, then the reproducing property of the kernel can be used. Thus it holds
\begin{align*}
P_N(x)^2 \leq& \min_{i=1,..,N} \Vert k(\cdot, x) - k(x_i,x) \cdot k(\cdot, x_i) \Vert_{\ns}^2 \\
=& \min_{i=1,..,N} k(x,x) - 2 \cdot k(x_i,x) \cdot k(x,x_i) + k(x_i,x)^2 \cdot \underbrace{k(x_i, x_i)}_{= \Phi(0) = 1} \\
=& \min_{i=1,..,N} 1-k(x_i,x)^2 \\
=& \min_{i=1,..,N} 1-\Phi(\Vert x-x_i \Vert)^2 .
\end{align*}
Theorem \ref{th:estimate_fill_dist} gives an estimate on the fill distance as $h_N \leq C \gamma^{-2/(\tau-d/2)} \cdot N^{-1/d}$. Thus for $N$ large enough it 
holds $h_N < r_0$, i.e., $\min_{i=1,..,N} \Vert x-x_i \Vert < h_N < r_0 ~ \forall x \in \Omega$, hence the monotonicity of $\Phi$ can be used in the following. 
\\
Theorem \ref{th:lower_bound_power_func} provides a lower bound on the decay of the Power function, i.e., $\Vert P_N \Vert_{L^{\infty}} \geq c_P \cdot N^{1/2 - 
\tau/d}$. Any $\gamma$-stabilized greedy algorithm restricts the point selection to $\Omega_\gamma^{(N)} \equiv \{x\in \Omega: P_N(x) \geq \gamma \cdot \Vert 
P_N \Vert_\infty \}$, therefore the distance from such a point $x \in \Omega_\gamma^{(N)}$ from the old points $x_i \in X_N$ can be estimated:
\begin{alignat}{1} 
\label{eq:proof_dist_Omega}
c_P^2 \cdot \left(N^{\frac{1}{2}-\frac{\tau}{d}}\right)^2 \leq& \Vert P_N \Vert_\infty^2 \leq \gamma^{-2} \cdot P_N(x)^2 \nonumber \\
\leq& \gamma^{-2} \cdot \min_{i=1,..,N} 1-\Phi(\Vert x-x_i \Vert)^2 \nonumber \\
\leq& \gamma^{-2} \cdot \min_{i=1,..,N} 2c \Vert x-x_i \Vert^b \nonumber \\
\Rightarrow ~ \left( \frac{c_P^2 \cdot \gamma^2}{2c} \right)^{1/b} \cdot \left( N^{\frac{1}{2}-\frac{\tau}{d}} \right)^{2/b} \leq& \min_{i=1,..,N} \Vert x-x_i 
\Vert
\end{alignat}
In the last step $1 - \Phi(\Vert x-x_i \Vert)^2 = (1+\Phi(\Vert x-x_i \Vert)) \cdot (1-\Phi(\Vert x-x_i \Vert)) $ was estimated from above with help of the 
assumption $1-\Phi(r) \leq c \cdot r^b$ and the estimate $\Phi(\Vert x-x_i \Vert) \leq \Phi(0) = 1$. %
Since $x \in \Omega_\gamma^{(N)}$ was arbitrary, Inequality \eqref{eq:proof_dist_Omega} is valid for all $x \in \Omega_\gamma^{(N)}$. Thus it holds
\begin{align*}
\text{dist}(\Omega_\gamma^{(N)}, X_N) =& \inf_{x \in \Omega_\gamma^{(N)}} \min_{i=1,..,N} \Vert x-x_i \Vert \\
\geq& \left( \frac{c_P^2 \cdot \gamma^2}{2c} \right)^{1/b} \cdot \left( N^{\frac{1}{2}-\frac{\tau}{d}} \right)^{2/b}.
\end{align*} 
Now define $C = c_P^{2/b} \cdot \left(2c \right)^{-1/b}$ and apply Lemma \ref{lem:dist_sep_estimate} which finishes the proof.
\end{proof}

\begin{rem}
Theorem \ref{th:estimate_dist_Omega} allows to derive statements about the separation distance for a larger class of kernels, since no further restriction 
$\tau > d/2 + 1$ is needed. 
Although the derived result is weaker, it is still optimal in certain cases, even not covered by the first theorem. For example, if one applies the theorem to 
the basic Mat{\'e}rn kernel whose radial basis function is given by $\Phi(r) = e^{-r}$, the estimate yields $q_{N+1} \geq C \cdot \gamma^{-2} \cdot N^{-1/d}$ 
which shows uniformity also for this kernel.
\end{rem}

\section{Stability and refined error estimates} \label{sec:extensions}
As anticipated, we can now use the results derived in the previous sections to draw some additional conclusions on the stability and convergence of the 
interpolation by any $\gamma$-stabilized algorithm. 

Despite all the following results are just the combination of known results with
properties of the selected points, it is worth mentioning them here to highlight the features of the new algorithm.
These features are indeed either refined versions of the ones proven in \cite{SH16b} for the $P$-greedy selection, or completely new results. In particular,
they prove for the first time that a kernel-greedy algorithm can achieve stability, provide error bounds w.r.t. to general $L^p$ norm for the approximation 
of derivatives, and approximate functions which are outside of the native space.

\begin{cor}
Assume that $k$ is a translational invariant kernel with finite smoothness $\tau$, and let  $\Omega \subset \mathbb{R}^d$ be a bounded domain which satisfies 
an interior cone condition. Let $(X_N)\subset\Omega$ be the sequence of sets of points selected by any $\gamma$-stabilized algorithm, and let $c, c'>0$ denote 
generic constants independent of $N$ and $\gamma$. 
\begin{enumerate}[label=\roman*.]
 \item\label{item:stability} Bounds on $\lambda_{\min}$: Assume additional smoothness $\tau > d/2 +1$. Then the minimal eigenvalue $\lambda_{\min}(X_N)$ of the kernel matrix of $k$ on 
$X_N$ can be bounded as
\begin{align*}
c \cdot \gamma^{8\tau - 4d -4} \cdot N^{1-2\tau/d} \leq \lambda_{\min}(X_N) \leq c' \cdot \gamma^{-4} \cdot N^{1-2\tau/d}
\end{align*}
\item\label{item:lebesgue} Bound on the Lebesgue constant: Assume additional smoothness $\tau > d/2 +1$. Then there exists $N_0\in\N$ and a constant $c>0$ such that for all $N\geq N_0$ the 
Lebesgue constant is bounded as
\begin{align*}
\Lambda_{X_N}\leq c \gamma^{-4(\tau - d/2)} \sqrt{N}. 
\end{align*}
\item\label{item:convergence} Bounds for derivatives: Under the conditions of Theorem \ref{th:estimate_derivatives_wendland}, there exists $N_0\in\N$ and a constant $c>0$ such that 
for all $f\in W_p^{\tau}(\Omega)$ and all $N\geq N_0$ it holds
\begin{align*}
\seminorm{W_q^m(\Omega)}{f - \Pi_{V(X_N)}f} \leq c \gamma^{\frac{-2\left(\tau-m-d\left(1/p-1/q\right)_+\right)}{\tau-d/2}} 
N^{\left(-\frac{\tau-m}{d}+\left(\frac1p-\frac1q\right)_+\right)} \norm{W_p^{\tau}(\Omega)}{f},
\end{align*}
where the bound holds in particular for all $f\in\ns$, by using a different constant, $p=2$, and the norm $\norm{\ns}{f}$ in the right hand side.
 
\item\label{item:uniform} Asymptotic point distribution: Assume additional smoothness $\tau > d/2 +1$. Then the sequence of points is asymptotically uniformly distributed, and in particular 
there exists a constant $c>0$ such that
\begin{align*}
\rho_{X_N}:=\frac{h_N}{q_N} \leq c\ \gamma^{-4} \fa N\in\N.
\end{align*}

 \item\label{item:escaping} Escaping the native space: For all  $d/2<\beta\leq \tau$ and for all $0\leq \mu\leq \beta$ there exists $N_0\in\N$ and a constant $c>0$ such that for all 
$f\in W_2^{\beta}(\Omega)$ and for all $N\geq N_0$ it holds
 \begin{align*}
\seminorm{W_2^{\mu}(\Omega)}{f - \Pi_{V(X_N)}f} 
&= c \gamma^{-\frac{2 (\beta-\mu)}{\tau-d/2}-4(\tau-\mu)} N^{-\frac{\beta-\mu}{d}}  \norm{W_2^{\beta}(\Omega)}{f}.
\end{align*}

\end{enumerate}
\end{cor}
\begin{proof}
We prove the five points separately:
\begin{enumerate}[label=\roman*.]
\item In the case of kernel of finite smoothness it is known that the smallest eigenvalue can be bounded from below with help of the separation distance $q_N$ 
by $\lambda_\text{min}(X_N) \geq \tilde{c} \cdot q_N^{2\tau - d}$, see e.g. \cite[Chap. 12.2]{Wendland2005}. The lower estimate on the seperation distance from 
Theorem \ref{th:lower_bound_sep_dist2} can be applied to conclude the stated lower estimate. 

For the upper estimate we use $\lambda_\text{min}(X_N) \leq | P_{N-1}(x_N)|^2$ \cite[Chap. 12.1]{Wendland2005} and estimate the right hand side further by 
applying the upper bound on the maximal value of the Power function from Theorem \ref{th:decay_power_func_restricted} which yields the result.

\item This is simply the application of point \eqref{item:uniform} to Theorem 1 of \cite{DeMarchi2010}.
\item Here we just substitute the decay of the fill distance of Theorem \ref{th:estimate_fill_dist} into the estimate of Theorem 
\ref{th:estimate_derivatives_wendland}. The rate of decay on the right hand side is thus given by
\begin{align*}
h_N^{\tau-m-d(1/p-1/q)_+}
&\leq\left( c \cdot \gamma^{-2/(\tau-d/2)} \cdot N^{-1/d}\right)^{\tau-m-d(1/p-1/q)_+}\\
&\leq c \gamma^{\frac{-2\left(\tau-m-d\left(1/p-1/q\right)_+\right)}{\tau-d/2}} N^{-\frac{\tau-m}{d}+\left(\frac1p-\frac1q\right)_+}.
\end{align*}
To obtain the bound for $f\in\ns$, it is sufficient to recall that $\ns$ is norm equivalent to $W_2^{\tau}(\Omega)$.

\item This points is just a direct computation using the estimates of Theorem \ref{th:estimate_fill_dist} and Theorem \ref{th:lower_bound_sep_dist2}.
\item Theorem 4.2 from \cite{Narcowich2006} gives that for all $f\in W_2^{\beta}(\Omega)$ with $d/2<\beta\leq \tau$, and for all $0\leq \mu\leq \beta$, there 
exists 
$h_0$ such that for all $X\subset\Omega$ with $h_X\leq h_0$ it holds
\begin{align*}
\seminorm{W_2^{\mu}(\Omega)}{f - \Pi_{V(X)}f} \leq C h_X ^{\beta-\mu} \rho_X ^{\tau-\mu} \norm{W_2^{\beta}(\Omega)}{f}.
\end{align*}
We can then use the bounds of Theorem \ref{th:estimate_fill_dist} and the bound on the uniformity constant proven in the previous point to obtain that
\begin{align*}
\seminorm{W_2^{\mu}(\Omega)}{f - \Pi_{V(X_N)}f} 
&\leq C' \gamma^{-\frac{2 (\beta-\mu)}{\tau-d/2}} \gamma^{-4(\tau-\mu)} N^{-\frac{\beta-\mu}{d}}  \norm{W_2^{\beta}(\Omega)}{f}\\
&= C' \gamma^{-\frac{2 (\beta-\mu)}{\tau-d/2}-4(\tau-\mu)} N^{-\frac{\beta-\mu}{d}}  \norm{W_2^{\beta}(\Omega)}{f}.
\end{align*}
\end{enumerate}
\end{proof}

Observe that point \eqref{item:stability} implies that using a $\gamma$-stabilized greedy algorithm yields a provable stable interpolation process, since the 
condition number of the interpolation matrix cannot grow arbitrary fast. Thus it is justified to name any $\gamma$-restricted algorithm 
also $\gamma$\textit{-stabilized} algorithm. 

\begin{rem}
We remark that precise error bounds w.r.t.\ an $L^q$ norm and for the approximation of derivatives are difficult to obtain having only bounds on the 
$L^{\infty}$ norm of the Power function as the ones of Theorem \ref{th:decay_power_func_restricted}. For example, if $k$ is smooth enough then any $f\in\calh$ 
is continuous, and 
thus if $\Omega$ is bounded it holds
\begin{align*}
\norm{L^q(\Omega)}{f - \Pi_{V(X_N)}(f)} 
&\leq  \norm{L^{\infty}(\Omega)}{f - \Pi_{V(X_N)}(f)} \meas(\Omega)^{1/q}\\
&\leq \meas(\Omega)^{1/q} \norm{L^{\infty}(\Omega)}{P_{N}} \norm{\ns}{f - \Pi_{V(X_N)}(f)} \\
&\leq c \meas(\Omega)^{1/q}  \gamma^{-2} \cdot N^{\frac{1}{2}-\frac{\tau}{d}}  \norm{\ns}{f - \Pi_{V(X_N)}(f)}.
\end{align*}
This technique, which was used in \cite{SH16b}, gives an extra factor of $N^{1/2 - (1/2 - 1/q)_+}$ which is in general suboptimal compared to the one of the 
point \eqref{item:convergence} of the corollary. 

This improvement is relevant also for other scenarios that have been discussed in the literature. For example, for $q = 2$ and $m = 0$ the estimate within 
point \eqref{item:convergence} of the corollary gives a rate of convergence of order 
$N^{-\frac{\tau}{d}}$, and not just $N^{\frac12-\frac{\tau}{d}}$.
This improved rate of convergence was observed numerically in \cite{Schaback2018a} in the setting of superconvergence with uniformly distributed points, though 
not proven (see the discussion at p. 21--22). 
\end{rem}

\section{Numerical experiments} \label{sec:numerical_experiments}

\subsection{Decay rates of Power function}

To complement the analytically derived upper and lower bound on the decay of the value of the Power function for any $\gamma$-stabilized algorithm, some 
numerical experiments are 
performed. 

To stay as general as possible, no concrete selection criterion was chosen, and instead arbitrary points within the restricted area $\Omega_\gamma^{(N)}$ were 
chosen. In practice this is done by selecting a random point within the restricted set and, to minimize the effect of this non-deterministic 
selection, every run was repeated $10$ times. However, we remark that the computed 
values still slightly depend on the computer that was used.

The experiments were performed on $\Omega = [0,1]^d$ for $d \in \{1, 3, 5\}$ using 
$\gamma \in \{0.1, 0.2, .., 0.9, 1\}$. The set $\Omega$ was discretized with $3 \cdot 10^4$ uniformly randomly sampled points.

The basic Mat\'ern and the linear Mat\'ern kernel were used, 
whose radial basis functions are given by
\begin{align*}
\Phi_{\text{bas}}(r) =& e^{-r} \\
\Phi_{\text{lin}}(r) =& (1+r) \cdot e^{-r}.
\end{align*}

Up to $N = 800$ sampling points were chosen in each run, and they correspond to up to $800$ values of the maximal value of the Power function which are denoted as 
$\{\Vert P_N 
\Vert_{L^\infty(\Omega)}\}_{N \in [1:800]}$.

To estimate the numerical rate of convergence, we first computed a fit to the values $\{\log(\Vert 
P_N \Vert_{L^\infty(\Omega)})\}_{N \in I }$ for $I \in \{[a, b] ~ | ~ a \in \{50, 75, 100\}, b \in \{600, 700, 800 \} \}$
with functions of the form $y_2(n) = \log(\alpha) + \lambda \cdot \log(n)$. Then, the mean value of these nine 
fit parameters $\alpha$ is regarded as a meaningful estimation of the numerical decay value. 
Using this procedure, we compensate the effect of the values of the Power function for the first iterates (which are affected the most by the non-deterministic 
point selection) and the ones for the final iterates. This mean value and the standard deviation of these nine values is displayed in Figure 
\ref{fig:convrates_powerfunc} for the different experimental settings. Additionally, the dotted line indicates the expected theoretical value of $1/2-\tau/d$.

We can confirm that the computed numerical values are in accordance with the analytical result. Taking into account that the numerically computed values are 
fit parameters, whereas the analytically derived quantity $1/2-\tau/d$ is the decay rate of an upper and lower bound, the results are matching.

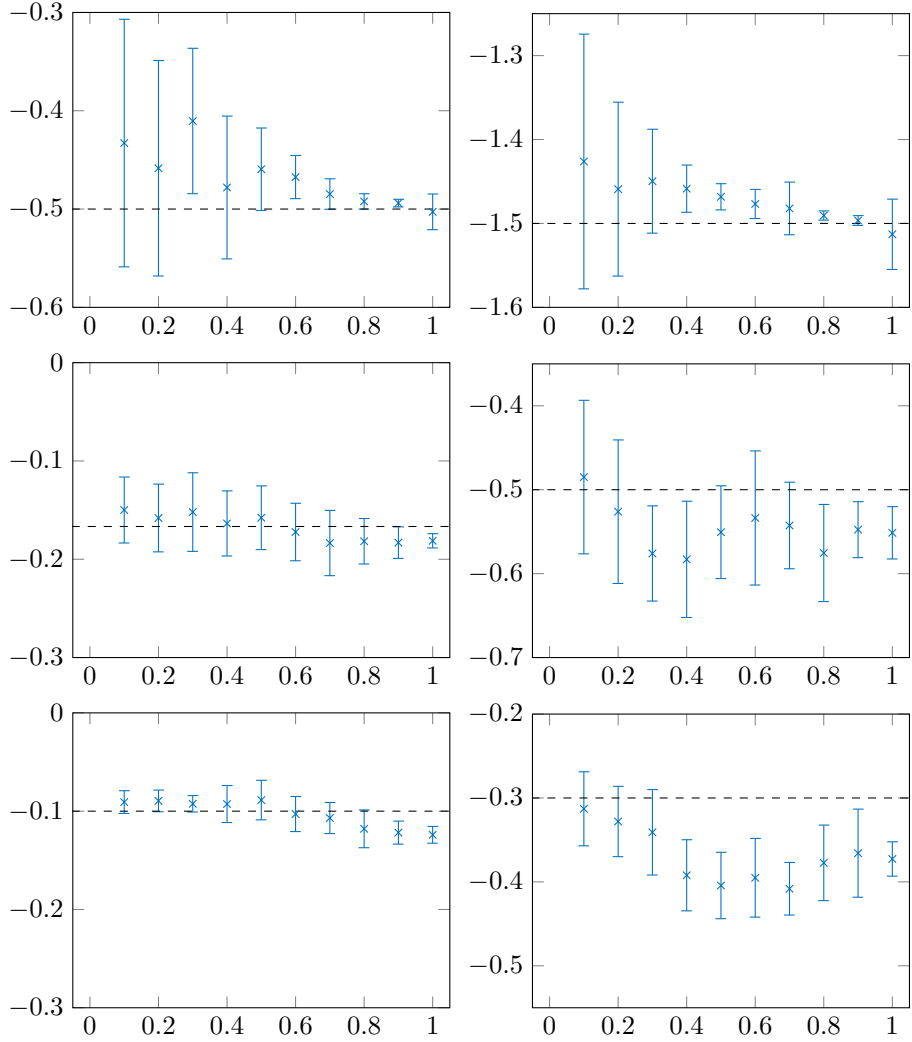
\begin{figure}[htp]
\setlength\fwidth{.43\textwidth}
\centering
%
%
\definecolor{mycolor1}{rgb}{0.00000,0.44700,0.74100}%
\begin{tikzpicture}

\begin{axis}[%
width=0.951\fwidth,
height=0.75\fwidth,
at={(0\fwidth,0\fwidth)},
scale only axis,
xmin=-0.05,
xmax=1.05,
ymin=-0.6,
ymax=-0.3,
axis background/.style={fill=white}
]
\addplot [color=mycolor1, draw=none, mark=x, mark options={solid, mycolor1}, forget plot]
 plot [error bars/.cd, y dir = both, y explicit]
 table[row sep=crcr, y error plus index=2, y error minus index=3]{%
0.1	-0.432848750298806	0.125943354399479	0.125943354399479\\
0.2	-0.458514372746713	0.109558098146562	0.109558098146562\\
0.3	-0.410399368593493	0.0739462236725471	0.0739462236725471\\
0.4	-0.478066517098345	0.0726553266214762	0.0726553266214762\\
0.5	-0.459478834287909	0.041930848227684	0.041930848227684\\
0.6	-0.467456665293884	0.0219561129972174	0.0219561129972174\\
0.7	-0.484722630716747	0.0154810450917111	0.0154810450917111\\
0.8	-0.492227303378184	0.00778454788835096	0.00778454788835096\\
0.9	-0.493937603769849	0.00388875251005832	0.00388875251005832\\
1	-0.502843759552397	0.0181583288687447	0.0181583288687447\\
};
\addplot [color=black, dashed, forget plot]
  table[row sep=crcr]{%
-0.1	-0.5\\
1.1	-0.5\\
};
\end{axis}

\end{tikzpicture}%
%
%
\definecolor{mycolor1}{rgb}{0.00000,0.44700,0.74100}%
\begin{tikzpicture}

\begin{axis}[%
width=0.951\fwidth,
height=0.747\fwidth,
at={(0\fwidth,0\fwidth)},
scale only axis,
xmin=-0.05,
xmax=1.05,
ymin=-1.6,
ymax=-1.25,
axis background/.style={fill=white}
]
\addplot [color=mycolor1, draw=none, mark=x, mark options={solid, mycolor1}, forget plot]
 plot [error bars/.cd, y dir = both, y explicit]
 table[row sep=crcr, y error plus index=2, y error minus index=3]{%
0.1	-1.42616624460492	0.151786300693509	0.151786300693509\\
0.2	-1.4591159348793	0.103611350885239	0.103611350885239\\
0.3	-1.44966478017449	0.0619022291276815	0.0619022291276815\\
0.4	-1.45856640983563	0.0281371072234157	0.0281371072234157\\
0.5	-1.46827122877057	0.0156206429756592	0.0156206429756592\\
0.6	-1.47679656388933	0.0173383599412206	0.0173383599412206\\
0.7	-1.4821289008878	0.0313994922764172	0.0313994922764172\\
0.8	-1.49056384257109	0.00546184751103261	0.00546184751103261\\
0.9	-1.49650669852995	0.00581807528615969	0.00581807528615969\\
1	-1.5129576334732	0.0418928176985486	0.0418928176985486\\
};
\addplot [color=black, dashed, forget plot]
  table[row sep=crcr]{%
-0.1	-1.5\\
1.1	-1.5\\
};
\end{axis}
\end{tikzpicture}%
%
%
\definecolor{mycolor1}{rgb}{0.00000,0.44700,0.74100}%
\begin{tikzpicture}

\begin{axis}[%
width=0.951\fwidth,
height=0.75\fwidth,
at={(0\fwidth,0\fwidth)},
scale only axis,
xmin=-0.05,
xmax=1.05,
ymin=-0.3,
ymax=0,
axis background/.style={fill=white}
]
\addplot [color=mycolor1, draw=none, mark=x, mark options={solid, mycolor1}, forget plot]
 plot [error bars/.cd, y dir = both, y explicit]
 table[row sep=crcr, y error plus index=2, y error minus index=3]{%
0.1	-0.14988754636359	0.0335660638636063	0.0335660638636063\\
0.2	-0.158008474346889	0.0344709185844447	0.0344709185844447\\
0.3	-0.151999827331408	0.0399623282218788	0.0399623282218788\\
0.4	-0.163559739551792	0.0331216850681091	0.0331216850681091\\
0.5	-0.157768718256209	0.0324035548665735	0.0324035548665735\\
0.6	-0.172282354854874	0.029248729918187	0.029248729918187\\
0.7	-0.183497325082881	0.0331942765579113	0.0331942765579113\\
0.8	-0.18167402207367	0.0231364165474408	0.0231364165474408\\
0.9	-0.18306942699247	0.0160473345057553	0.0160473345057553\\
1	-0.181216236613093	0.00727188474334162	0.00727188474334162\\
};
\addplot [color=black, dashed, forget plot]
  table[row sep=crcr]{%
-0.1	-0.166666666666667\\
1.1	-0.166666666666667\\
};
\end{axis}
\end{tikzpicture}%
%
%
\definecolor{mycolor1}{rgb}{0.00000,0.44700,0.74100}%
\begin{tikzpicture}

\begin{axis}[%
width=0.951\fwidth,
height=0.747\fwidth,
at={(0\fwidth,0\fwidth)},
scale only axis,
xmin=-0.05,
xmax=1.05,
ymin=-0.7,
ymax=-0.35,
axis background/.style={fill=white}
]
\addplot [color=mycolor1, draw=none, mark=x, mark options={solid, mycolor1}, forget plot]
 plot [error bars/.cd, y dir = both, y explicit]
 table[row sep=crcr, y error plus index=2, y error minus index=3]{%
0.1	-0.484913041493392	0.0914797122649271	0.0914797122649271\\
0.2	-0.526188873996523	0.0855405558269543	0.0855405558269543\\
0.3	-0.575950851769848	0.056721582355975	0.056721582355975\\
0.4	-0.582982155184392	0.0692966810332161	0.0692966810332161\\
0.5	-0.550602455346132	0.0552686774271147	0.0552686774271147\\
0.6	-0.533656129546527	0.0799318801065413	0.0799318801065413\\
0.7	-0.542618837457652	0.0515754636742958	0.0515754636742958\\
0.8	-0.575421965907576	0.0578878796662733	0.0578878796662733\\
0.9	-0.547592381760488	0.0333500295409031	0.0333500295409031\\
1	-0.551348353770155	0.0311567019200202	0.0311567019200202\\
};
\addplot [color=black, dashed, forget plot]
  table[row sep=crcr]{%
-0.1	-0.5\\
1.1	-0.5\\
};
\end{axis}
\end{tikzpicture}%
%
%
\definecolor{mycolor1}{rgb}{0.00000,0.44700,0.74100}%
\begin{tikzpicture}

\begin{axis}[%
width=0.951\fwidth,
height=0.75\fwidth,
at={(0\fwidth,0\fwidth)},
scale only axis,
xmin=-0.05,
xmax=1.05,
ymin=-0.3,
ymax=0,
axis background/.style={fill=white}
]
\addplot [color=mycolor1, draw=none, mark=x, mark options={solid, mycolor1}, forget plot]
 plot [error bars/.cd, y dir = both, y explicit]
 table[row sep=crcr, y error plus index=2, y error minus index=3]{%
0.1	-0.090744466543422	0.0115405118016188	0.0115405118016188\\
0.2	-0.0895256958586424	0.0110048899752175	0.0110048899752175\\
0.3	-0.0925007668855116	0.00836865722525682	0.00836865722525682\\
0.4	-0.092729287964312	0.0188192479794409	0.0188192479794409\\
0.5	-0.0887144477671652	0.0201195867995896	0.0201195867995896\\
0.6	-0.102921855084782	0.0178317608072174	0.0178317608072174\\
0.7	-0.106946977977516	0.0158140714479308	0.0158140714479308\\
0.8	-0.117959214528656	0.0192315396973658	0.0192315396973658\\
0.9	-0.121815467273983	0.0117087703292071	0.0117087703292071\\
1	-0.124071307039827	0.00851404292809513	0.00851404292809513\\
};
\addplot [color=black, dashed, forget plot]
  table[row sep=crcr]{%
-0.1	-0.1\\
1.1	-0.1\\
};
\end{axis}

\end{tikzpicture}%
%
%
\definecolor{mycolor1}{rgb}{0.00000,0.44700,0.74100}%
\begin{tikzpicture}

\begin{axis}[%
width=0.951\fwidth,
height=0.747\fwidth,
at={(0\fwidth,0\fwidth)},
scale only axis,
xmin=-0.05,
xmax=1.05,
ymin=-0.55,
ymax=-0.2,
axis background/.style={fill=white}
]
\addplot [color=mycolor1, draw=none, mark=x, mark options={solid, mycolor1}, forget plot]
 plot [error bars/.cd, y dir = both, y explicit]
 table[row sep=crcr, y error plus index=2, y error minus index=3]{%
0.1	-0.31291556969091	0.0441114265530284	0.0441114265530284\\
0.2	-0.328034252350561	0.0419120701162572	0.0419120701162572\\
0.3	-0.340913400419332	0.0509067508857036	0.0509067508857036\\
0.4	-0.392130602090383	0.0423682293294549	0.0423682293294549\\
0.5	-0.404303733410405	0.039571777478897	0.039571777478897\\
0.6	-0.395093141410368	0.046901574449215	0.046901574449215\\
0.7	-0.40819413887513	0.0313367447724738	0.0313367447724738\\
0.8	-0.377321835274497	0.0450326664487731	0.0450326664487731\\
0.9	-0.365791777097728	0.0524817794992628	0.0524817794992628\\
1	-0.372701156356239	0.0204713829616534	0.0204713829616534\\
};
\addplot [color=black, dashed, forget plot]
  table[row sep=crcr]{%
-0.1	-0.3\\
1.1	-0.3\\
};
\end{axis}
\end{tikzpicture}%
\caption[Convergence rates for the decay of the Power function.]{Convergence rates for the decay of the maximal values $\Vert P_N \Vert_{L^\infty(\Omega)}$ of 
the Power function according to stabilized greedy algorithm for varying restriction parameters $\gamma$ applied to $\Omega = [0,1]^d$. The horizontal axis 
describes the restriction parameter $\gamma$ and the vertical axis describes the convergence rate. \\
From top to bottom the dimensions $d=1,3,5$ were used. On the left side the unscaled basic Mat\'ern kernel was applied, on the right side the unscaled linear 
Mat\'ern kernel was used. The computed decay rates are scattered around a mean. The dashed line marks the proven decay rate of $1/2 - \tau/d$. }
\label{fig:convrates_powerfunc}
\end{figure}

\subsection{Improved accuracy of \emph{f/P}-greedy}

In the following, an example is provided which shows that the $\gamma$-stabilized $f/P$-greedy algorithm is able to yield a better interpolant than the plain 
$f/P$-greedy algorithm in terms of expansion size and accuracy. 

For this, functions of the form 
\begin{align*}
f_\alpha: [-0.5, 0.5] \rightarrow \mathbb{R}, x \mapsto |x|^\alpha \cdot e^{-x^2}
\end{align*}
are considered which depend on a parameter $\alpha>0$. A calculation of the Fourier transform of $f_\alpha: \mathbb{R} \rightarrow \mathbb{R}$ shows 
that $f_\alpha \in 
W_2^{\alpha + 1/2-\epsilon}(\mathbb{R}) ~ \forall \epsilon > 0$ and thus we also have $f_\alpha \in W_2^{\alpha + 1/2 - \epsilon}(\R) ~ \forall \epsilon > 0$.

The function $f_\alpha$ is investigated with the $\gamma$-stabilized $f/P$-greedy algorithm and the unstabilized $f/P$-greedy algorithm using the linear 
Mat\'ern kernel. 
The native space of the linear Mat\'ern kernel on $[-0.5, 0.5]$ is norm equivalent to $W_2^2([-0.5, 0.5])$. Thus $\alpha = 1.51$ is chosen, since it yields a 
function with 
smoothness almost $2.01$, which is close to the smoothness of the native space. Furthermore, $\alpha = 3.5$ is chosen, which yields a function with 
approximately double 
the smoothness of the native space.

The training and the test set consists of each $10^5$ uniformly sampled distinct points within $\Omega$. The algorithm was run until the condition number of the 
kernel matrix exceeded a bound of $10^{14}$. The application of the $\gamma$-stabilized $f/P$-greedy algorithms for $\gamma \in \{10^{-4}, 10^{-3}, 10^{-2}, 
10^{-1}, 10^0\}$ as well as the unstabilized $f/P$-greedy yielded different decays of the residuals. Figure \ref{fig:special_functions_2} displays the results 
and Table \ref{tab:special_function_2_accuracy} collects some numbers, namely the number $N_{\max}$ of chosen sampling points and the remaining error, i.e., 
the maximal error of the final residual $\Vert r_{N_{\max}} \Vert_\infty$.

From the values in Table \ref{tab:special_function_2_accuracy} and Figure \ref{fig:special_functions_2} we can draw the following conclusions:
\begin{enumerate}[label=\roman*.]
\setlength\itemsep{0em}
\item In the case of $\alpha = 3.5, \gamma = 0$ the interpolant explodes due to numerical inaccuracies. Apart from this the unstabilized $f/P$-greedy algorithm 
(i.e., $\gamma=0$) and the barely stabilized ones yields the fastest residual decay rates such that their lines partly hide each other in the first figure. The 
reason is presumably that the restriction parameter is that small, that the limitation due to $\Omega_\gamma^{(N)}$ is not an actual restriction and thus almost 
the same points are chosen all the time.

With increasing value of $\gamma$ the decay behaviours become worse, but they are always between the decay behaviours for the unstabilized greedy ($\gamma=0$) and the $P$-greedy 
($\gamma=1$).
\item Increasing the value of the restriction parameter $\gamma$ yields more sampling points. This can be observed best within the table.

For intermediate values of $\gamma$ similar accuracies compared to $\gamma=1$ can be achieved but with way less sampling points. This speeds up the evaluation 
of the interpolant.
\end{enumerate}
Thus there is a tradeoff between speed of the residual decay and the amount of chosen sampling points. In fact for $\alpha = 3.5$ the minimal error is attained 
using an intermediate restriction parameter, namely $\gamma = 10^{-3}$. The reason is that the unstabilized $f/P$-greedy algorithm and the one with a very small 
restriction parameter $\gamma$ stop too early due to the criterion on the condition number. With increasing restriction parameter $\gamma$ more and more points 
are chosen which yield a better approximation. 

As a result it can be concluded that the $\gamma$-stabilized algorithms yield also in practice a more stabilized algorithm, such that more points can be 
selected. 
Especially for the $\gamma$-stabilized $f/P$-greedy the restriction yields a better approximation, since the algorithm does not stop that early due to the 
exceed of the bound on the condition number. The additional sampling points help to reach a better approximation. Since the $f/P$ selection criterion 
incorporates the Power function in the denominator, the sampling points are sometimes chosen close nearby already selected points. 

Moreover, when these stability conditions are satisfied, it is evident that the use of a function-dependent selection rule improves the convergence speed with 
respect to a pure $P$-greedy selection. This effect, as mentioned before, is not appearing in our estimates, but it is an expected behavior that we plan to 
investigate further. 

For any practical use of the algorithm this tradeoff can be faced by selecting the restriction parameter $\gamma$ with $K$-fold cross validation, so that an 
almost optimal value is used for a given problem. \\

\begin{figure}[ht]
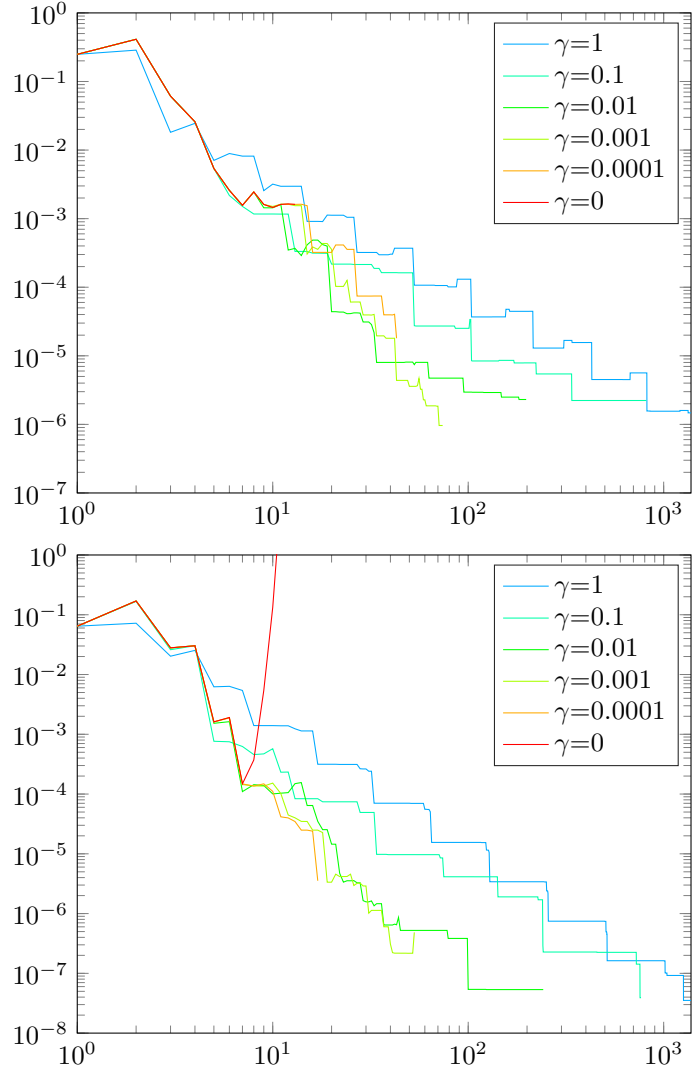

\setlength\fwidth{.7\textwidth}
\centering
\input{matlab/example_improved_accuracy_1.51_fP.tex}
\input{matlab/example_improved_accuracy_3.5_fP.tex}
\caption{Residual decay for functions given by $f_\alpha(x) = |x|^\alpha \cdot \exp(-x^2)$ for the linear Mat\'ern kernel for $\alpha=1.51$ (top) and 
$\alpha=3.5$ (bottom) using the $\gamma$-stabilized $f/P$-greedy. The horizontal axis describes the number of chosen sampling points, the vertical axis 
describes the maximal residual value. For small values of $\gamma$ or $\gamma=0$ the algorithm stopped early, since the condition number exceeded a predefined 
bound of $10^{14}$.} 
\label{fig:special_functions_2}
\end{figure}

\begin{table}[htp]
\centering
\caption[Comparison of errors.]{Minimal error as well as number of chosen points depending on the restriction parameter $\gamma$ for the function $f_\alpha$ 
with $\alpha \in \{1.51, 3.5 \}$.} %
\renewcommand{\arraystretch}{1.2}
\begin{tabular}{|l|ll|ll|} \hline
$f/P$ 			& \multicolumn{2}{|c|}{$\alpha=1.51$} 	& \multicolumn{2}{|c|}{$\alpha=3.5$} \\ \hline \hline
$\gamma = $	& $N_{\max}$	& $\Vert r_{N_{\max}} \Vert_\infty$ 	& $N_{\max}$	& $\Vert r_{N_{\max}} \Vert_\infty$ \\ \hline
$0$			& 13	& $1.61 \cdot 10^{-3}$	& 29	& $4.88 \cdot 10^{31}$ \\
$10^{-4}$	& 43	& $1.80 \cdot 10^{-5}$	& 17	& $3.54 \cdot 10^{-6}$ \\
$10^{-3}$	& 74	& $9.63 \cdot 10^{-7}$	& 53	& $4.85 \cdot 10^{-7}$ \\
$10^{-2}$	& 198	& $2.31 \cdot 10^{-6}$	& 242	& $5.35 \cdot 10^{-8}$ \\
$10^{-1}$	& 815	& $2.23 \cdot 10^{-6}$	& 766	& $3.91 \cdot 10^{-8}$ \\
$10^{0}$	& 1380	& $1.47 \cdot 10^{-6}$	& 1380	& $3.53 \cdot 10^{-8}$ \\ \hline
\end{tabular}
\label{tab:special_function_2_accuracy}
\end{table}

\subsection{Distribution of sampling points}

Finally we provide examples to visualize the distribution of the selected sampling points on some arbitrarily chosen domain for several values of $\gamma$. 
For this we consider
\begin{align*}
g(\varphi) = 0.35 \cdot (\cos \left( \pi \left( \frac{\varphi}{\pi} \right)^2 \right)+2) \cdot (0.15 \cdot \cos(\varphi)^2 + 0.3).
\end{align*}
Let $\theta(x) \in [0, 2\pi)$ denote the standard angular coordinate of $x \in \R^2$ in the polar coordinate system and define $a := (0.17, 0.17)^T \in \R^2$. 
The domain $\Omega \subset \R^2$ is given by $\Omega = \Omega_1 \setminus \Omega_2$ with
\begin{align*}
\Omega_1 &:= \{ x + (0.1, 0)^T \in \R^2 ~ | ~ \Vert x \Vert_2 < g(\theta(x) + \pi) \} \\
\Omega_2 &:= \{ x \in \R^2 ~ | ~ \Vert x-a \Vert_2^2 \leq 0.003  \}.
\end{align*}
For the numerical implementation of the greedy selection algorithm, the domain $\Omega$ was discretized with 831 uniformly randomly distributed points. Now the 
linear Mat\'ern kernel and the $\gamma$-stabilized $f/P$-greedy algorithm with $\gamma \in \{0, 0.04, 0.15, 1\}$ was applied to the function 
\begin{align*}
f(x) = \frac{1}{\Vert x-a \Vert^2}
\end{align*}
and 50 sampling points were selected for each $\gamma$-value. The different point distributions are displayed in Figure \ref{fig:example_point_distribution}. One can observe that for $\gamma=0$ all the points are clustered close to the point $a$, which is the midpoint of the "cut out" hole $\Omega_2$. For $\gamma=1$ the points are uniformly distributed within $\Omega$. The cases $\gamma=0.04$ and $\gamma=0.15$ provide intermediate distributions.

\begin{figure}[htp]
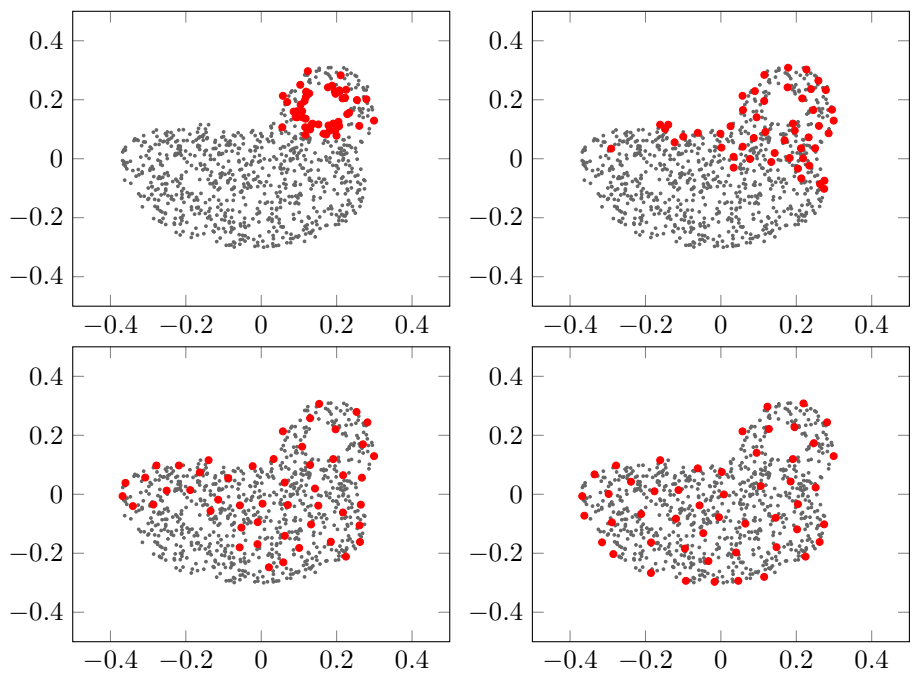

\setlength\fwidth{.43\textwidth}
\centering
\input{matlab/example_point_distribution_0.tex}
\input{matlab/example_point_distribution_0.04.tex}
\input{matlab/example_point_distribution_0.15.tex}
\input{matlab/example_point_distribution_1.tex}
\caption[Point distribution for several $\gamma$-values.]{Selected sampling points (red) of the $\gamma$-stabilized $f/P$-greedy algorithm using the linear 
Mat\'ern kernel. From left to right, top to bottom $\gamma \in \{ 0, 0.04, 0.15, 1 \}$ was used. For small $\gamma$ values the point distribution is better 
adopted to the given data function $f$.}
\label{fig:example_point_distribution}
\end{figure}

\section{Conclusion and Outlook}

In this paper the class of $\gamma$-stabilized greedy kernel algorithms was introduced and investigated analytically, especially for kernels of finite smoothness $\tau > d/2$. 
Precise estimates for the decay of the Power function were derived, and these lead to convergence rates. Furthermore, the resulting point distribution was 
quantified and their asymptotical uniform distribution was proven. This strong result leads to improved convergence rates and stability statements. The results 
were illustrated with numerical examples.

However, some questions remain open. Notably, though the $\gamma$-stabilized algorithm can be specialized to single functions, we derived in this paper only 
general convergence statements that hold for the entire set of functions of the native space.
Refined statements on the 
adaptive behavior which could be observed in the numerical examples for specific functions 
  will be the focus of further investigations.

\vspace{.5cm}
\textbf{Acknowledgements:} The authors acknowledge the funding of the project by the Deutsche Forschungsgemeinschaft (DFG, German Research Foundation) under
Germany's Excellence Strategy - EXC 2075 - 390740016. \\
We also thank Dominik Wittwar for discussions leading to the idea of the stabilized algorithms.

\bibliography{bib_arxiv}
\bibliographystyle{abbrv}

\end{document}